\documentclass[11pt]{amsart}
\usepackage{amsmath,amssymb,txfonts}
\usepackage{amssymb}
\usepackage{amsxtra}
\usepackage{amsmath}
\usepackage{txfonts}

\usepackage[cp1252]{inputenc}

\usepackage{mathrsfs}
\usepackage{graphicx}

\usepackage[active]{srcltx}

\usepackage[notref,notcite]{}

\def\rr{{\mathbb R}}
\def\rn{{{\rr}^n}}

\def\zz{{\mathbb Z}}
\def\nn{{\mathbb N}}
\def\hh{{\mathbb H}}

\def\fz{\infty}
\def\az{\alpha}

\def\dist{{\mathop\mathrm{\,dist\,}}}
\def\loc{{\mathop\mathrm{\,loc\,}}}

\def\lz{\lambda}
\def\dz{\delta}

\def\ez{\epsilon}

\def\bz{\beta}

\def\gz{{\gamma}}

\def\boz{{\Omega}}

\def\sz{\sigma}

\def\wz{\widetilde}

\def\ls{\lesssim}
\def\gs{\gtrsim}

\def\bint{{\ifinner\rlap{\bf\kern.35em--}
\int\else\rlap{\bf\kern.45em--}\int\fi}\ignorespaces}

\def\bbint{{\ifinner\rlap{\bf\kern.35em--}
\hspace{0.078cm}\int\else\rlap{\bf\kern.45em--}\int\fi}\ignorespaces}

\def\esup{\mathop\mathrm{\,esssup\,}}
\def\einf{\mathop\mathrm{\,essinf\,}}

\def\diam{{\mathop\mathrm{\,diam\,}}}

\def\r{\right}
\def\lf{\left}

\newtheorem{thm}{Theorem}[section]
\newtheorem{lem}[thm]{Lemma}
\newtheorem{prop}[thm]{Proposition}
\newtheorem{rem}[thm]{Remark}
\newtheorem{cor}[thm]{Corollary}
\newtheorem{defn}[thm]{Definition}
\numberwithin{equation}{section}

\subjclass[2000]{42B35; 46E30; 47B38; 30H25}
\keywords{Quasiconformal mappings; Compositions; Q-spaces}

\begin{document}
\allowdisplaybreaks
\arraycolsep=1pt

\title[A Quasiconformal Composition Problem for the $Q$-spaces]{A Quasiconformal Composition Problem for the $Q$-spaces}
\author{Pekka Koskela, Jie Xiao, Yi Ru-Ya Zhang and Yuan Zhou}

\address{Pekka Koskela:\ Department of Mathematics and Statistics,
         P.O. Box 35 (MaD), FI-40014, University of Jyv\"askyl\"a, Finland}
         \email{pkoskela@maths.jyu.fi}

\address{Jie Xiao:\ Department of Mathematics and Statistics,
         Memorial University,
         NL A1C 5S7, Canada}
         \email{jxiao@mun.ca}

\address{Yi Ru-Ya Zhang:\  Department of Mathematics, Beijing University of Aeronautics and Astronautics, Beijing 100191, P.R. China
and Department of Mathematics and Statistics,
         P.O. Box 35 (MaD), FI-40014, University of Jyv\"askyl\"a, Finland}
          \email{ddfx023zy@gmail.com}

          \address{Yuan Zhou:\  Department of Mathematics, Beijing University of Aeronautics and Astronautics, Beijing 100191, P.R. China}
                    \email{yuanzhou@buaa.edu.cn}

\thanks{Pekka Koskela was supported by the Academy of Finland grant 120972; Jie Xiao was supported by NSERC of Canada (\# 202979463102000) and URP of Memorial University (\# 208227463102000); Yi Zhang and Yuan Zhou were supported by
the New Teachers' Fund for Doctor Stations (\# 20121102120031) and Program for New Century Excellent Talents in University (\# NCET-11-0782)  of Ministry of
Education of China, and National Natural Science Foundation of China (\# 11201015) .}

\begin{abstract}
Given a quasiconformal mapping $f:\rn\to\rn$ with $n\ge2$, we show that
(un-)boundedness of the composition operator ${\bf C}_f$
on the spaces $Q_{\az}(\rn)$
depends  on the index
 $\az$ and
the degeneracy set of the Jacobian $J_f$.
We establish  sharp results in terms of the index $\az$ and the local/global self-similar Minkowski dimension of the degeneracy set of  $J_f$.
This gives a solution to \cite[Problem 8.4]{ejpx} and also reveals a completely new phenomenon,  which is totally different from
the known results for Sobolev, BMO, Triebel-Lizorkin and Besov spaces. Consequently,
Tukia-V\"ais\"al\"a's quasiconformal extension  $f:\rn\to\rn$ of an arbitrary quasisymmetric mapping $g:\rr^{n-p}\to \rr^{n-p}$
is shown to preserve  $Q_{\az} (\rn)$ for any $(\az,p)\in (0,1)\times[2,n)\cup(0,1/2)\times\{1\}$. Moreover, $Q_{\az}(\rn)$ is shown to be invariant under inversions for all $0<\az<1$.
\end{abstract}
 \maketitle

\tableofcontents

 \section{Introduction}\label{s1}

Quasiconformal mappings can be characterized via
invariant function spaces. For example, a homeomorphism $f:\rn\to\rn,$ $n\ge 2,$
is quasiconformal if and only if the composition operator ${\bf C}_f$ (given by ${\bf C}_f(u)=u\circ f$) is
bounded on
the homogeneous Sobolev space $\dot W^{1,\,n}(\rn)$; see for example \cite{k11}.
The composition property is easiest seen from the usual
analytic definition, according to which a homeomorphism $f:\rn\to\rn,$ $n\ge 2,$
is quasiconformal if $f\in W^{1,\,1}_\loc(\rn;\rn)$ and there is a constant $K\ge 1$ so that
$$|Df(x)|^n \le K J_f(x),\quad a.\,e.\ x\in\rn.$$
Indeed, modulo technicalities, one simply uses the chain rule and a change
of variables. It is far less obvious that also the invariance of the
Triebel-Lizorkin spaces $\dot F^s_{n/s,\,q}(\rn)$ with $0<s<1$ and
$n/(n+s)<q<\infty$
characterizes quasiconformality, see \cite{v89,bp03,kyz,kkss}. The difficulty
here is that one has to deal with ``fractional derivatives'' and thus the
inequality from the analytic definition is not immediately helpful.
For the off-diagonal Besov spaces $\dot B^s_{n/s,\,q}(\rn)$ with $q\ne n/s$,
the situation is different:
 each homeomorphism $f$ for which  ${\bf C}_f$ is bounded on
$\dot B^s_{n/s,\,q}(\rn)$ has to be quasiconformal and even
bi-Lipschitz; these spaces are clearly bi-Lipschitz invariant, see
\cite{kkss}. Recall here
that $f$ is bi-Lipschitz if there exists a constant $L\ge1$ such that
$$\frac1 L|x-y|\le |f(x)-f(y)|\le L|x-y|,\quad \forall x,\,y\in\rn.$$
Furthermore, the John-Nirenberg space $BMO(\rn)$ is invariant under
quasiconformal mappings and
each sufficiently regular homeomorphism $f$ for which ${\bf C}_f$ is a bounded
operator on $BMO(\rn)$  is necessarily quasiconformal; see \cite{r74,a}.

In their 2000 paper \cite{ejpx}, Essen, Jasson, Peng and Xiao introduced the so-called Q-spaces $Q_\az(\rn),$ $0<\alpha <1,$ that satisfy
$$
\dot W^{1,\,n}(\rn)\subset\dot F^{\az}_{n/\az,\,n/\az}(\rn)\subseteq Q_\az(\rn)\subseteq BMO(\rn).
$$
Each $Q_\az(\rn)$ consists of all $u\in L^2_\loc(\rn)$ with
\begin{eqnarray*}
\|u \|_{Q_\az(\rn)}&&=\sup_{x_0\in\rn,\,r>0} \lf(r^{2\az-n}\int_{B(x_0,\,r)}\int_{B(x_0,\,r)}\frac{|u(x)-u(y)|^{2}}{|x-y|^{n+2\az}}\,dx\,dy\r)^{1/2}<\infty.
\end{eqnarray*}
The above definition actually makes perfect sense for all
$-\infty<\alpha<\infty,$ but the case $\alpha\ge 1$ (when $n\ge 2$) reduces to constant
functions and the case $\alpha<0$ to $BMO(\rn)$; see \cite{ejpx}.
These spaces have received considerable interest. In \cite{ejpx}, five open
problems related to the spaces $Q_\az(\rn)$ were posed. All but the following
one of them have by now been solved.

\bigskip

\noindent {\bf A quasiconformal composition problem for the $Q$-spaces} (\cite[Problem 8.4]{ejpx}):\ \ {\it Let $f$ be a quasiconformal mapping.
 Prove or disprove the boundedness of the composition operator ${\bf C}_f$ on $Q_\alpha(\rn)$ with $\alpha\in (0,1)$.}

\bigskip

By the above string of inclusions of function spaces, all of which except
for the $Q$-spaces are known to be quasiconformally invariant, suggests that
the answer should be in the positive.

We show that, surprisingly, the answer to the above question
depends on the quasiconformal
mapping in question through the shrinking properties of the mapping.
For example, the quasiconformal mapping $f(x)=x|x|$ induces a bounded
composition operator for all $0<\alpha<1,$ but if the Jacobian of a
quasiconformal mapping decays to zero when we approach a sufficiently large
set, then the invariance may fail. Thus, the case of
$Q$-spaces is be very different from the other function spaces
that we discussed above.

In order to state our results,
we need to introduce some terminology whose analogues have appeared in
estimating the upper box-counting dimension of the singular set of a
suitable weak solution of the Navier-Stokes system \cite{RoSa}.

\begin{defn}\label{d1.1} \rm For a set $E\subseteq\rn$ and every $r>0$, denote by $N_{\rm cov}(r, E)$ the minimal number of cubes with edge length $r$ required to cover $E$.

\item\rm(i) The local self-similar Minkowski dimension of $E$ is defined as
\begin{equation}\label{e1.2}\overline\dim_L\, E=  \liminf _{N\to\fz}  \limsup_{r\to0}
\sup_{B\subset\rn\atop Nr\le r_B\le1}\frac{\log N_{\rm cov}(r,\,E\cap B)}{ \log(r_B  /r )},
\end{equation}
where  the supremum is taken over all balls $B=B(x_B,\,r_B)\subset\rn$ with  $ r_B\in [Nr,\,1]$.

\item\rm(ii) The global self-similar Minkowski dimension of $E$ is defined as
\begin{equation}\label{e1.3}\overline\dim_{LG}\,E= \liminf _{N\to\fz}  \sup_{r>0}
\sup_{B\subset\rn\atop r_B\ge Nr}\frac{\log N_{\rm cov}(r,\,E\cap B)}{ \log(r_B/r )},
\end{equation}
where the first supremum is taken over all
$r\in(0,\,\fz)$ and the second  is
 over all balls $B=B(x_B,\,r_B)\subset\rn$ with $r_B\in [Nr,\,\fz)$.
\end{defn}

We also need the concept of the local Muckenhoupt class.

\begin{defn}\label{d1.1d} \rm
For a closed set $E\subseteq\rn$ and a nonnegative function $w:\rn\to\rr$,
we say that $w$ belongs to the local Muckenhoupt class $A_1(\rn;\,E)$
provided there exists a positive constant $C$ such that
\begin{equation}\label{e1.1}
 \bint_{B}w(z)\,dz\le C\,\einf_{x\in B} w(x)
 \end{equation}
holds for every ball $B=B(x_B,\,r_B)\subset\rn$ with $2r_B<d(x_B,\,E)$. Naturally, $A_1(\rn;\,\emptyset)$ stands for the Muckenhoupt class $A_1(\rn)$. Accordingly, $E$ is called the degeneracy set of $w$ when  $w\in A_1(\rn;\,E)$.
\end{defn}

The main result of this paper is the following theorem.

\begin{thm}\label{t1.2}
Given $n\ge 2$, let $f:\rn\to\rn$ be a quasiconformal mapping with $J_f\in A_1(\rn;\, E)$ for some closed set $E\subseteq\rn$.
If $E$ is a bounded set with $\overline\dim_L\, E\in[0,\,n)$ or $E$ is an unbounded set with $\overline\dim_{LG}\,E\in[0,\,n)$,
 then   ${\bf C}_f$ is bounded on $Q_{\az}(\rn)$ for all
\begin{eqnarray}\label{e1.xxx}
0<\az<\lf\{\begin{array}{ll}
  \min \{ 1 ,\frac {n-\overline\dim_L\,E} 2\} ,& \ {\rm if} \ E\ {\rm is\ bounded};\\
  \min \{ 1 ,\frac {n-\overline\dim_{LG}\,E}2 \} ,\quad & \ {\rm if} \ E\ {\rm is\ unbounded}.
\end{array}\r.
\end{eqnarray}

 In particular, if  $E$ is a bounded set with  $\overline\dim_L\, E\in[0,\,n-2]$ or
 $E$ is a unbounded set with   $\overline\dim_{LG}\, E\in[0,\,n-2]$, then  ${\bf C}_f$ is bounded on $Q_{\az}(\rn)$ for all $\az\in(0, 1)$.

 \end{thm}

Theorem \ref{t1.2} is essentially sharp, see Theorems \ref{t1.4} and 1.7 below.

As the first important consequence of Theorem \ref{t1.2},  we have the following result.

\begin{cor}\label{c1.3} Let $\alpha\in (0,1)$ and $ 0\ne\bz\in\rr$. If $f(z)=|z|^{\bz-1}z$, then ${\bf C}_f$ is bounded on $Q_{\az}(\rn)$. In particular, $Q_\az{(\rn)}$ is conformally invariant in the sense that $g\in Q_{\az}(\rn)$ if and only if $x\mapsto g(x|x|^{-2})$ is in $Q_{\az}(\rn)$.
\end{cor}

Furthermore, for the Tukia-V\"ais\"al\"a quasiconformal extension $f: \ \rn\to\rn$ of an arbitrary quasiconformal (quasisymmetric) mapping $g:\rr^{n-p}\to\rr^{n-p}$, we obtain the second important consequence of Theorem \ref{t1.2}.

\begin{cor} \label{c1.x3} Given $1\le p<n$, suppose $g:\ \rr^{n-p}\to\rr^{n-p}$ is a quasiconformal mapping when $n-p\ge 2$, or a quasisymmetric mapping when $n-p=1$.
Let $f: \ \rn\to\rn$ be the Tukia-V\"ais\"al\"a's quasiconformal extension of  $g$ as in  \cite{tv}.  Then
the following hold:

\item\rm{(i)} $J_f,\, J_{f^{-1}}\in A_1(\rn;\rr^{n-p})$;

\item\rm{(ii)} ${\bf C}_f,\,{\bf C}_{f^{-1}}$ are bounded on $Q_{\az}(\rn)$ for all
$$0<\az<
\begin{cases}
\frac12 \quad\hbox{when}\quad p=1,\\
1\quad\hbox{when}\quad p\ge 2.
\end{cases}
$$
Consequently, $u\in Q_\az(\rn)$ if and only if $u\circ f\in Q_\az(\rn)$.
\end{cor}

The proof of Theorem \ref{t1.2} relies on a new characterization of $Q$-spaces
established in Section 3. This technical result allows us to employ our
Muckenhoupt assumption
and the  control on the number of Whitney-type balls guaranteed by
our dimension estimate. We expect that our approach will allow one to handle
various other function spaces as well.

Our assumption on the control of the fractal size of the degerancy set, whenever which is bounded or unbounded,  is
necessary in the following sense.

 \begin{thm} \label{t1.4}

 Let $n\ge2$ and $0<\az_0<1$. There is a bounded set $E_{\az_0}$ with $\overline\dim_LE_{\az_0}=n-2\az_0$ and a quasiconformal  (Lipschitz) mapping
 $f:\rn\to\rn$ with $J_f\in A_1(\rn; E_{\az_0})$ for which ${\bf C}_f$ is not bounded on $Q_\az(\rn)$ for any $\az\in(\az_0,\,1)$.
 \end{thm}

The main idea in the constructions for Theorem \ref{t1.4} is to patch up suitable pieces of radial stretchings in a family of pairwise disjoint balls. In this manner,
we also construct an unbounded set $\wz E_{\az_0}\subset\zz^n$ with $\overline\dim_{LG}{\wz E}_{\az_0}=n-2\az$
but $\overline\dim_L\wz E_{\az_0}=0$ and an associated quasiconformal mapping as in Theorem \ref{t1.4}; see below. This also shows the need for $\overline\dim_{LG}$ in Theorem 1.3.

\begin{thm}\label{t5.1}
 Let $n\ge2$ and $0<\az_0<1$. There exists a unbounded set   $\wz E_{\az_0}\subset \zz^n$ with $\overline\dim_{LG}\wz E_{\az_0}=n-2\az_0$ but $\overline\dim_L\wz E_{\az_0}=0$,
and a quasiconformal  (Lipschitz) mapping
 $f:\rn\to\rn$ with $J_f\in A_1(\rn; \wz E_{\az_0})$ for which ${\bf C}_f$ is not bounded on $Q_\az(\rn)$ for any $\az\in(\az_0,\,1)$.
\end{thm}

This paper is organized as follows:
Section \ref{s2} clarifies the relationship between the Minkowski dimension and the local Minkowski dimension $\overline\dim_L$ or the global Minkowski dimension $\overline\dim_{LG}$
and also computes  $\overline\dim_L$ and $\overline\dim_{LG}$ for the sets in  Theorems \ref{t1.4} and \ref{t5.1};
Section \ref{s3} explores a new aspect of $Q_\alpha(\rn)$, which  will be used in the proof of Theorem \ref{t1.2}; in Section \ref{s4}, we prove   Theorem \ref{t1.2};
Section 5 contains the proofs of Corollaries \ref{c1.3} and \ref{c1.x3}; Section \ref{s5} is devoted to the proofs of Theorems \ref{t1.4} and  \ref{t5.1}.

Finally, as the converse of the above open question, given a homeomorphism  $f:\rn\to\rn$ for which  the composition operator ${\bf C}_f$ is a bounded  on ${Q}_\az(\rn)$
for some $\az\in(0,\,1)$,
one would like to know if $f$ is necessarily quasiconformal. The answer is actually in the positive, at least under suitable regularity assumptions on the homeomorphism in question. Since this requires some work, the details will be given in a forthcoming paper.

{\it Notation.}\ \ In the sequel, we denote by $C$ a positive constant which is independent of the main parameters, but may vary from line to line. The symbol $A\ls B$ or $B\gs A$
means that $A\le CB$. If $A\ls B$ and $B\ls A$, we then
write $A\sim B$. For any locally integrable function $u$ and measurable set $X$,
we denote by $\bbint_X u $ the average
of $u$ on $X$, namely, $\bbint_X u  \equiv\frac 1{|X|}\int_Xu\,dx$.
For a set $\boz$ and $x\in\rn$, we use $d(x,\,\boz)$ to denote $\inf_{z\in\boz}|x-z|$, the distance from $x$ to $\boz$.
For $\lambda Q$, we mean   the cube  concentric with $Q$, with sides parallel to the axes, and with length
$\ell(\lambda Q) =\lambda \ell(Q)$; similarly,   $\lz B$ denotes the ball  concentric with $Q$ with radius $\lz r_B$, where $r_B$ is the radius of $B$.

\section{Local and global Minkowski dimensions}\label{s2}

In this section, we  clarify the relation between the Minkowski dimension and the above dimensions $\overline\dim_L$ and $\overline\dim_{LG}$.
Recall that for a bounded set $E\subset\rn$, its Minkowski dimension $\dim_ME$ is defined by
$$\overline\dim_M\, E =\limsup_{r\to0} \frac{\log N_{\rm cov}(r,\,E )}{ \log(1/r )},$$
 where $N_{\rm cov}(r,\,E )$ is the minimum number of cubes with edge length $r$
 required  to cover $E$.

\begin{lem}\label{lemma 2.1}

\item\rm(i) For every set $E\subset\rn$ and every $R\ge 1$,  we have
 $$\overline{\dim}_L\,E =\liminf_{N\to\fz}\limsup_{r\to 0}\sup_{{B\subset\rn} \atop {Nr\le r_B\le R}}\frac{\log N_{\rm cov}(r,E\cap B)}{\log(r_B/r)}.$$

\item\rm(ii) For every set $E\subset\rn$,
we always have
$$0 \le \sup_B\overline{\dim}_M (E\cap B)\le \overline{\dim}_L\, E \le \overline{\dim}_{LG}\,E \le n,$$
where the supremum is taken over all balls in $\rn$.

\item\rm(iii) If $E\subset F$, then $\overline{\dim}_L\, E\le \overline{\dim}_L\,F $
and $\overline{\dim}_{LG}\,E\le \overline{\dim}_{LG}\,F $.

\end{lem}

\begin{proof}
(i) From the definition, we always have
$$\overline{\dim}_L\,E \le\liminf_{N\to\fz}\limsup_{r\to 0}\sup_{{B\subset\rn} \atop {Nr\le r_B\le R}}\frac{\log N_{\rm cov}(r,E\cap B)}{\log(r_B/r)}.
$$
Towards the reverse inequality, notice that
every ball $B$ of radius $1\le r_B\le R$ can be covered by $c_n R^n$  balls $B_i$ of radii $1$. So
 $$N_{\rm cov}(r,E\cap B)\le c_n R^n\sup  \{N_{\rm cov}(r,E\cap \wz B):\ {\wz B\subset \rn} \ {\rm with} \ {r_{\wz B}=1}\}$$
 and hence for all $r<r_B/N$ and $r<1$, we have
$$\frac{\log N_{\rm cov}(r,E\cap B)}{\log{r_B/r}}\le
\frac{\log c_n R^n}{\log N}+ \sup_{{\wz B\subset \rn}\atop{r_{\wz B}=1}} \frac{\log N_{\rm cov}(r,E\cap \wz B)}{\log(1/r)}.$$
Since the first term on the right-hand side tends to $0$ as $N\to\fz$, by the definition of $\overline\dim_L E$, we obtain the desired inequality.

(ii) Obviously, $ \overline{\dim}_{LG}(E)\le n $ is obtained from  $N_{\rm cov}(r,\,E\cap B) \le (2 r_B/r)^n$
for every ball $B$ with radius  $r_B\ge Nr$;
indeed,
$$\overline\dim_{LG}\,E\le \liminf _{N\to\fz}  \sup_{r>0}
\sup_{B\subset\rn\atop r_B\ge Nr}\frac{n\log(r_B  /r )+n}{ n\log(r_B  /r ) }
=  \liminf _{N\to\fz} \frac{n\log N  +n}{ n\log N }=n.
$$
The other inequalities follow  from the definitions and (i) directly.

(iii) These statements are trivial.
\end{proof}

If $E$ is a set of finitely many points, observing that
$N_{\rm cov}(r,\,E\cap B) \ls 1$
for every ball $B$ with radius  $r_B\ge Nr$,  we obtain
$$\overline\dim_{LG}\,E\le
  \liminf _{N\to\fz} \frac{ \log C}{  \log N }=0,
$$
which implies that
$$\overline{\dim}_M\,E = \overline{\dim}_L\,E = \overline{\dim}_{LG}\,E=0.$$

However, for a countable set $E$,
 $\overline \dim_{LG}$, $\overline \dim_L$ and
 $\sup_B \overline\dim_M(E\cap B)
$ may be very different.
Write
 $ ( \nn )^n  = \nn  \times\cdots\times  {\nn }$
  and
$ (2^\nn)^n=2^\nn\times\cdots\times 2^\nn$
with  $2^\nn=\{2^k: \,k\in\nn\}$.
For  $\theta\in[0,\,1]$, set
\begin{equation}\label{e2.y1}2^{\nn_\theta}:=\bigcup_{k\in\nn\cup\{0\}}A_{k,\,\theta}:= \bigcup_{k\in\nn\cup\{0\}}\{ 2^k,\,2^{k}+1,\,\,\cdots,\, 2^{k}+ 2^{[\theta k]} \}.
\end{equation}
where $ [  \theta k ] $  is the largest integer less than or equal to $\theta k$.
Write $(2^{\nn_\theta})^n= 2^{\nn_\theta}\times\cdots\times 2^{\nn_\theta}$.
Observe that
$$
\begin{cases}
2^{\nn_\theta}=\nn\cup\{0\}\ \ \hbox{when}\ \ \theta=1;\\
2^{\nn_\theta}=2^{\nn}\cup\{1\} \ \ \hbox{when}\ \ \theta=0.
\end{cases}
$$
We always have $$\overline \dim_M ((2^{\nn_\theta})^n\cap B)= \overline \dim_L((2^{\nn_\theta})^n\cap B)=\overline {\rm dim}_{LG} ((2^{\nn_\theta})^n\cap B) =0$$ for all balls $B$ and all $\theta\in[0,\,1]$
since $(2^{\nn_\theta})^n\cap B$ only contains finitely many points.

\begin{lem}\label{lemma 2.2}
Let  $\theta\in[0,\,1]$. Then
\begin{equation}\label{equality 2.1}\overline\dim_{LG} (2^{\nn_\theta})^n  =\theta n ;
\end{equation}
in particular,
 $ \overline\dim_{LG} (2^\nn)^n  =0 $ and  $\overline\dim_{LG}  \nn ^n  =n  $.
But
$\overline\dim_{L }(2^{\nn_\theta})^n= 0.$
\end{lem}
\begin{proof}
We first show that $\overline\dim_{L}\, (2^{\nn_\theta})^n=0.$
Observe that each  $B\subset\rn$ with $ r_B\le 1$ contains at most a uniform  number  of points in $\zz^n$.
So for each $N\ge1$ and $r\in(0,\,r_B/N)$, we can cover  $B\cap (2^{\nn_\theta})^n$ by a uniform  number  of balls of radii $r $,
that is,
  $N_{\rm cov}(1,\,(2^{\nn_\theta})^n\cap B)\ls 1$,  which implies that
 $\dim_{L}\,(2^{\nn_\theta})^n\le 0 $ by definition. So
 by Lemma \ref{lemma 2.1}, $\overline\dim_{L}\,(2^{\nn_\theta})^n=0.$

To show \eqref{equality 2.1}, we first consider the easy cases $\overline\dim_{LG}\nn ^n =n  $
and  $ \overline\dim_{LG}(2^\nn)^n =0 $.
Indeed, for  every ball $B
\subset\rn$ with $r_B= N$, we have
 $$N_{\rm cov}(1,\,\nn^n\cap B)= \sharp (\nn^n\cap B)\ge ( N/\sqrt n )^n,$$
 which implies that
$$\overline\dim_{LG}\,\nn ^n\ge   \liminf _{N\to\fz}
 \frac{\log (N/\sqrt n)^n}{ \log N   }=n,$$
 and hence, by Lemma \ref{lemma 2.1}, $\overline\dim_{LG} \nn ^n =n.$

On the other hand, for each $N$ and $r>0$, if $r\le 1$  and $Nr<r_B$, we have
$$N_{\rm cov}(r,\,(2^\nn)^n \cap B)\le  (\log  r_B )^n  ;$$
 if $2^k<r\le  2^{k+1}$ for some $k\ge 0$, we have
 $$N_{\rm cov}(r,\,(2^\nn)^n\cap B)\le \sqrt n [\log (r_B/r+2)]^n. $$
Hence
$$\overline\dim_{LG}\,(2^\nn)^n\le \liminf _{N\to\fz}  \sup_{r\ge 1}
\sup_{B,\,r_B\ge Nr}\frac{ n\log [\sqrt n\log (r_B/r+2)]  }{ \log(r_B/r )}=0.
$$
So by Lemma \ref{lemma 2.1}, we have $ \overline\dim_{LG}\,(2^\nn)^n =0.$

Generally, we let $\theta\in(0,\,1)$.
For every ball $B=B(0,\,\sqrt n 2^{m+1})$ with $m\ge 2/\theta+1$, we have
 $$N_{\rm cov}(1,\, (2^{\nn_\theta})^n \cap B))\ge \sharp [(2^{\nn_\theta})^n \cap B]\ge  2^{n\theta m} $$
an hence,
$$\overline\dim_{LG} (2^{\nn_\theta})^n \ge \liminf _{N\to\fz}\sup_{2^{m+1}\ge N}
 \frac{    n\theta m }{      (m+1)  }=\liminf _{N\to\fz}
 \frac{    n\theta (\log N)-n\theta }{     \log N  } =n\theta.$$

The proof of  $\overline\dim_{LG} (2^{\nn_\theta})^n \le  \theta n$
 is reduced to verifying that
for every large $N$,  all $r>0$ and all balls $B$ with $r_B\ge Nr $, we have
\begin{equation}\label{e2.2}N_{\rm cov}(r,\, (2^{\nn_\theta})^n\cap B))
\ls ( r_B/r)^{\theta n }.\end{equation}
Indeed, this implies that
\begin{eqnarray*}
\overline \dim_{LG}  (2^{\nn_\theta})^n  &&\le\liminf_{N\to\fz  }\sup_{r>0}
\sup_{B,\,r_B\ge Nr}\frac{  \log  [(r_B/r )^{\theta n }]+\log C  }{ \log (r_B/r)  }\\
&&=\liminf_{N\to\fz  }\frac{ (\theta n )\log  N +  \theta n \log C }{  \log N}\\
&&=\theta n.
\end{eqnarray*}

To prove \eqref{e2.2}, we consider two cases under the assumption $N\ge 2^5$.

{\bf Case 1}:\ $0<r\le1$. If $r_B<2$, then $(2^{\nn_\theta})^n\cap B$ contains no more than  a  uniform number of points and hence
$$
\sharp ((2^{\nn_\theta})^n\cap B)\ls 1\ls(r_B/r)^{\theta n}.
 $$
If  $2^m< r_B\le 2^{m+1}$ for some $m> 1$, then $(2^{\nn_\theta})^n\cap B\subset [0,\,2^{m+2}]^n$.
 Notice that the interval $[0,\,2^{m+2}]$ contains at most $\sum_{k=1}^{m+1}2^{\theta k}\sim 2^{\theta m}$
 points of $2^{\nn_\theta}$, and so we have
 $$
\sharp ((2^{\nn_\theta})^n\cap B)\ls 2^{\theta mn}\ls (r_B/r)^{\theta n},
 $$
 which implies that
$$N_{\rm cov}(r,\, (2^{\nn_\theta})^n\cap B) \le
\sharp ((2^{\nn_\theta})^n\cap B)\ls (r_B/r)^{\theta n}.
 $$

{\bf Case 2}:\ $r>1$.  Assume that $2^\ell<r<2^{\ell+1}$.
Given a ball $B$ with $r_B\ge Nr $, assume that
   $2^m< r_B\le 2^{m+1}$ for some $m\ge 5+\ell$. Then $(2^{\nn_\theta})^n\cap B\subset [0,\,2^{m+2}]^n$.
Observe that  $[0,\,2^\ell]$ can be covered by an interval  of  length $r$.
If $\ell \le k\le [\ell/\theta] $, then $\{2^k,\,2^k+1,\,\cdots,\,2^k+2^{[\theta k]}\}$ can be covered by  an interval  of  length $r$.
If  $k>[\ell/\theta] $, then  $\{2^k,\,2^k+1,\,\cdots,\,2^k+2^{[\theta k]}\}$ can be covered by  $2^{[\theta k]-\ell}+1$
 intervals  of  length $r$.
Thus when  $m\le [\ell/\theta]-2$, $ 2^{\nn_\theta} \cap [0,\,2^{m+2}]$ can be covered by
$m-\ell+2$   intervals  of  length $r$.
If $m> [\ell/\theta]-2$, then $ 2^{\nn_\theta} \cap [0,\,2^{m+2}]$ can be covered by
$2^{\theta m-\ell}$  intervals  of  length $r$. In both cases,
$ 2^{\nn_\theta} \cap [0,\,2^{m+2}]$ can be covered by $C2^{\theta(m-\ell)}\le C (r_B/r)^{\theta }$   intervals  of  length $r$.
Therefore
$$N_{\rm cov}(r,\, (2^{\nn_\theta})^n\cap [0,\,2^{m+2}]^n) \ls (r_B/r)^{\theta n }
 $$
which gives \eqref{equality 2.1} as desired.
\end{proof}

\begin{rem}\rm
Lemma \ref{lemma 2.2} indicates that
the dimension $\overline \dim_{LG}$ not only measures the local self-similarity and local Minkowski
size   but also measures the global  selfsimilarity  of $E$.
\end{rem}

For  a slight modification of the standard Cantor construction, we obtain $E_a$
and its self-similar extension $\mathcal E_a$ so that
$\overline\dim_L$ and $\overline\dim_{LG}$ are the same and coincide  with
$\dim_M E_a $.
Precisely, the sets $E_a$ and $\mathcal E_a$ are defined as follows.
Let $ a  \in(0,\,1)$.
Let  $I_{i}$,  $i=1,2$,
be the two closed intervals obtained by removing the middle open interval of length $a$ from $I_0=[0,\,1]$
 ordered from left to right;
 when $m\ge2$,
the subintervals $I_{i_1\cdots i_m}$, $i_m=1,2$,
are the two closed intervals  obtained by removing
the middle open intervals   of length $a [(1-a)/2]^{m-1}$ from $I_{ i_1\cdots i_{m-1}}$ ordered from left to right.  Notice that $|I_{ i_1\cdots i_{m }}|=[(1-a)/2]^{m }$
for $m\ge 1$.
For each $m\ge1$, set
$$
I^m_a=\bigcup_{i_1,\,\cdots,i_m\in\{1,2\}}{I_{ i_1\cdots i_{m }}}\ \ \&
\ \ E^m_a= (I^m_a)^n= I^m_a\times\cdots\times I^m_a.
$$
Notice that $E^m_a$ consists of $2^{mn}$ disjoint cubes $\{Q_{m,j}\}_{j=1}^{2^{mn}}$ with edge length $[(1-a)/2]^m$, and $E^{m+1}_a\subset E^{m}_a$.
Denote by $z_{m,\,j}$ the center of $Q_{m,j}$
and  $z_0= ({\frac12},\,\cdots,\,\frac12)$  the center of $Q_0=I_0^n$.
Denote by $E_a$ the closure of  the collection of all these centers, that is,
\begin{equation}\label{e2.y2}
 E_a=\overline{\{z_0,\, z_{m,\,j}:\,   m\in\nn,\,j=1,\,\cdots,\,2^{mn} \}}.
 \end{equation}
 Set \begin{equation}\label{e2.y3}\mathcal E_a=\bigcup_{k\ge 0}\lf\{\lf(\frac2{1-a}\r)^k x:\ x\in E_a\r\}.  \end{equation}
In this case, we consider the larger family $\{\wz Q_{m,j}\}_{m\in\zz,\,j\in\nn}$ consisting of all    $$\lf\{\lf(\frac2{1-a}\r )^k x:\ x\in Q_{m+k,\,i} \r\}$$ for all possible  $k\ge-m$ and $ i=1,\,\cdots,2^{(m+k)n}$.
Let $\wz z_{m,\,j}$ be the center of $\wz Q_{m,j}$. We also have
\begin{equation}
 \mathcal E_a=\overline{\{ \wz z_{m,\,j}:\,   m\in\zz,\,j\in\nn \}}.
 \end{equation}

\begin{lem}\label{lemma 2.4}
For every $a\in(0,\,1)$,
$$\overline\dim_M E_a=\overline\dim_L E_a=\overline\dim_L \mathcal E_a=
\overline\dim_{LG} E_a=\overline\dim_{LG} \mathcal E_a=\frac n{\log [2/(1-a)]}.$$
\end{lem}

\begin{proof}
By Lemma \ref{lemma 2.1}, it suffices to show that
$$
\overline\dim_M E_a\ge \frac n{\log [2/(1-a)]}\ \ \&\ \
\overline\dim_L \mathcal E_a\le \frac n{\log [2/(1-a)]}.
$$
To this end,
notice that for each $k>m$, we have
$$2^{(k-m)n}<N_{\rm cov}([1-a)/2]^k,\,E_a\cap
\wz Q_{m,j} )\le 2^{(k-m)n}+\sum_{\ell=m}^{k-1}2^{\ell n}< 2^{(k+1-m)n},$$
where recall that $\ell(\wz Q_{m,j})=[(1-a)/2]^m$.
For each $r<[(1-a)/2]^{m+2}  $, picking
$k_r>m$ such that  $$ [ (1-a)/2]^{k_r}<r\le  [ (1-a)/2]^{k_r-1},$$
we have
$$N_{\rm cov}([1-a)/2]^{k_r+1},\,\mathcal E_a\cap\wz Q_{m,j})\le N_{\rm cov}(r,\,\mathcal E_a\cap\wz Q_{m,j})< N_{\rm cov}([1-a)/2]^{k_r},\,\mathcal E_a\cap\wz Q_{m,j}),$$
and hence,
$  N_{\rm cov}(r,\,\mathcal E_a )\sim   2^{(k_r-m)n}.$
In particular,  $ N_{\rm cov}(r,\,E_a )\gs    2^{k_rn},$
which
 implies that $$\overline\dim_M E_a\ge\limsup_{r\to0} \frac{ k_rn +\log C}{\log (1/r)}=
\limsup_{k_r\to\fz} \frac{ k_rn +\log C}{ k_r\log[2/(1-a)]+\log C_1}=\frac n{\log [2/(1-a)]}.$$
Moreover, for each ball $B$ with  $  r_B\ge [(1-a)/2]^{ 3}r$,
there exists a $k_B\le k_\ez-2$ such that
$$
[ (1-a)/2]^{k_B}<r_B\le [ (1-a)/2]^{k_B-1}.
$$
Hence
$$N_{\rm cov}(r,\,\mathcal E_a\cap
B)\le
N_{\rm cov}([1-a)/2]^{k_r},\,\mathcal E_a\cap
B)\ls 2^{(k_r-k_B)n}.$$
Thus
\begin{eqnarray*}
&&\sup_{B,\,r_B\ge [(1-a)/2]^{-N}r}\frac{\log N_{\rm cov}(r,\,\mathcal E_a\cap
B) }{ \log(r_B/r)}\\
&&\quad\le \sup_{0\le m\le k_r-N }
\frac{ \log  C_12^{(k_r-m)n}  }{ \log   [(1-a)/2]^{m- k_r}}\\
&&\quad\le \sup_{0\le m\le k_r-N}
\frac{ n( k_r-m) +\log C_1    }{ ( k_r-m )\log  [2/(1-a) ] }\\
&&\quad\le
\frac{ nN +\log C_1    }{  N\log  [2/(1-a) ] }\\
&&\quad\to n/\log [2/(1-a)]\ \ \hbox{as}\ \ N\to\fz.
\end{eqnarray*}
Consequently, we get
$$
\overline \dim_{LG} \mathcal E_a \le  \frac{n}{\log [2/(1-a)]},
$$
as desired.
\end{proof}

\section{A characterization of $Q$-spaces} \label {s3}

In this section, we   characterize membership in $Q$-spaces via oscillations.
To do so, let us introduce a couple of concepts. Let $u$ be a measurable function. For $\az\in(0,\,1)$, $q\in(0,\,\fz)$,  and each ball $B=B(x_0,\,r)\subset\mathbb R^n$, set
$$
\Psi_{\az,\,q}(u,\,B)=\sum_{k\ge0} 2^{ 2k\az }\bint_{B(x_0,\,r)}
 \inf_{c\in\rr}\lf\{\bint_{  B(x ,\,2^{-k}r )}
|u(z)-c|^q\,dz\r\}^{2/q}\,dx.
$$
Define the space $Q_{\az,\,q}(\rn)$
as the collection of $u\in L^q_\loc(\rn)$ such that
$$\|u\| _{Q_{\az,\,q}(\rn)} =\sup_{x_0\in\rn, \,r>0}\big[\Psi_{\az,\,q}(u,\,B(x_0,\,r))\big]^{1/2} < \infty.$$
Also, for every ball $B\subset\rn$ and each function $u$ on $B$, set
$$
\Phi_\az (u,\,B)=|B|^{ 2\az/n-1}\int_B\int_B \frac{|u(x)-u(y)|^{2}}{|x-y|^{n+2\az}}\,dx\,dy.
$$
Then
$
\|u\| _{Q_{\az}(\rn)}=\sup_B\big[\Phi_\az (u,\,B)\big]^{1/2},
$
where the supremum is taken over all balls $B\subset\mathbb R^n$.

\begin{prop}\label{p3.1} Let $\az\in(0,\,1)$ and $q\in(0,\,2]$.
 There exists a constant $C$ such that for all measurable functions $u$ and all balls $B=B(x_0,\,r)$ one has
$$
 C^{-1}\Phi_\az (u,\,B(x_0,\,r/16))\le  \Psi_{\az,\,q} (u,\,B(x_0,\, r))\le C \Phi_\az (u,\,B(x_0,\,16r)).
$$
Consequently,
$$
Q_\az(\rn)=Q_{\az,\,q}(\rn)\quad\hbox{with}\quad\|\cdot\|_{Q_\az(\rn)}\sim \|\cdot\|_{Q_{\az,q}(\rn)}.
$$
\end{prop}

To verify Proposition \ref{p3.1}, we need the following estimate from \cite{kyz}.

\begin{lem}\label{lemma 3.2}
Let $\sz\in(0,\,\fz)$ and $u\in L^\sz_\loc(\rn)$. Then there is a set $E$ with $|E|=0$ such that for each pair of points
$x,\,y\in \rn\setminus E$  with $|x-y|\in[2^{-k-1}, 2^{-k})$
one has
\begin{eqnarray} \label{equation 3.1}
&&|u(x)-u(y)| \\
&&\quad\lesssim\sum_{j\ge k-2} \lf\{\inf_{c\in\rr}\lf[ \bint_{B(x,\,2^{-j})}|u(w)-c|^\sz \,dw\r]^{1/\sz }
+\inf_{c\in\rr}\lf[ \bint_{B(y,\,2^{-j})}|u(w)-c|^\sz \,dw\r]^{1/\sz }\r\}.\nonumber
\end{eqnarray}
\end{lem}

\begin{proof}[Proof of Proposition \ref{p3.1}] By Lemma \ref{lemma 3.2}, we obtain
\begin{eqnarray*}
&&\int_{B(x ,\,2r)}\frac{|u(x)-u(y)|^{2}}{|x-y|^{n+2\az}}\,dy\\
&&\quad\le
 \sum_{j=-1}^\fz (2^{-j}r)^{-(n+2\az)}\int_{B(x ,\,2^{-j}r)\setminus B(x ,\,2^{-j-1}r)}|u(x)-u(y)|^{2}\,dy\\
&&\quad\ls  \sum_{j=-1}^\fz (2^{-j}r)^{-2\az }\lf[\sum_{k\ge j-2}  \inf_{c\in\rr} \bint_{B(x,\,2^{-k}r)}|u(w)-c|^q \,dw\r]^{2/q }\\
  &&\quad\quad+  \sum_{j=-1}^\fz (2^{-j}r)^{-2\az} \bint_{B(x ,\,2^{-j}r)\setminus B(x ,\,2^{-j-1}r)}\lf[ \sum_{k\ge j-2}\inf_{c\in\rr}\bint_{B(y,\,2^{-k})}|u(w)-c|^q \,dw\r]^{2/q }\,dy \\
&&\quad=J_1(x)+J_2(x).
\end{eqnarray*}
Applying H\"older's inequality and changing the order of summation, we obtain
\begin{eqnarray*}
J_1(x)&&\ls
\sum_{j=-1}^\fz (2^{-j}r)^{-2\az }2^{-j\az}\sum_{k\ge j-2} 2^{k\az } \inf_{c\in\rr}\lf[ \bint_{B(x,\,2^{-k}r)}|u(w)-c|^q \,dw\r]^{2/q }\\
&&\ls\sum_{k\ge -3} 2^{k\az }
\sum_{j=-1}^{k} (2^{-j}r)^{-2\az } 2^{-j\az }\inf_{c\in\rr}\lf[ \bint_{B(x,\,2^{-k}r)}|u(w)-c|^q \,dw\r]^{2/q }\\
&&\ls\sum_{k\ge-3}(2^{-k}r)^{-2\az}
\inf_{c\in\rr}\lf[ \bint_{B(x,\,2^{-k}r)}|u(w)-c|^q \,dw\r]^{2/q }.
\end{eqnarray*}
Thus,
\begin{eqnarray*}
 r^{2\az-n}\int_{B(x_0,\,r)} J_1(x)\,dx
&&\ls  \sum_{k\ge-3} 2^{ 2k\az} \bint_{B(x_0 ,\,8r)}
 \inf_{c\in\rr}\lf[ \bint_{B(x,\,2^{-k}r)}|u(w)-c|^q \,dw\r]^{2/q }\,dx\\
&&\ls \Psi_{\az,\,q}(u,\,B(x_0,\,8r)).
 \end{eqnarray*}
For $J_2$, notice that
$$
\int_{B(x_0,\,r)}J_2(x)\,dx\ls \int_{B(x_0,\,4r)}\sum_{j=-1}^\fz (2^{-j}r)^{-2\az} \lf[ \sum_{k\ge j-2}\inf_{c\in\rr}\bint_{B(y,\,2^{-k})}|u(w)-c|^q \,dw\r]^{2/q }\,dy.
$$
Then, applying an argument similar  to the above estimate for $J_1$,  we have
\begin{eqnarray*}
 r^{2\az-n}\int_{B(x_0,\,r)} J_2(x)\,dx\ls \Psi_{\az,\,q}(u,\,B(x_0,\,8r)).
 \end{eqnarray*}
Combining  the estimates on $J_1$ and $J_2$, we obtain
$$
\Phi_{\az }(u,\,B(x_0,\,r))\ls \Psi_{\az,\,q}(u,\,B(x_0,\,8r)).
$$

On the other hand, noticing that for all $x\in\rn$, $r>0$ and $k\ge 0$ one has
$$
2^{-k }r\le |x-w|-|x-z|\le|z-w|\le |x-w|+|x-z|\le 2^{-k+3}r
$$
whenever
$$
z\in B(x,\,2^{-k}r)\ \ \&\ \ w\in B(x,\,2^{-k+2}r)\setminus B(x,\,2^{-k+1}r),
$$
we utilize $q\in (0,2]$ and the H\"older inequality to achieve
\begin{eqnarray*}
&& \inf_{c\in\rr}\lf[ \bint_{B(x,\,2^{-k}r)}|u(w)-c|^q \,dw\r]^{2/q}\\
&&\quad\ls\bint_{B(x,\,2^{-k}r)}|u(z)-u_{B(x,\,2^{-k}r)}|^2\,dz\\
&&\quad\ls \bint_{B(x,\,2^{-k}r)}|u(z)-u_{B(x,\,2^{-k+2}r)\setminus B(x,\,2^{-k+1}r)}|^2\,dz\\
&&\quad\ls \bint_{B(x,\,2^{-k}r)}\bint_{B(x,\,2^{-k+2}r)\setminus B(x,\,2^{-k+1}r)} |u(z)-u(w)|^2 \,dw\,dz\\
&&\quad\ls (2^{-k}r)^{2\az}\bint_{B(x,\,2^{-k}r)}\int_{B(x,\,2^{-k+2}r)\setminus B(x,\,2^{-k+1}r)}\frac{|u(z)-u(w)|^2}{|z-w|^{n+2\az}}\,dw\,dz\\
&&\quad\ls (2^{-k}r)^{2\az}\bint_{B(x,\,2^{-k}r)}\int_{B(z,\,2^{-k+3}r)\setminus B(z,\,2^{-k}r)}\frac{|u(z)-u(w)|^2}{|z-w|^{n+2\az}}\,dw\,dz.
\end{eqnarray*}
Thus, by changing the order of the integrals with respect to $dz$ and $dx$,
\begin{eqnarray*}
&&\Psi_{\az,\,q}(u,\,B(x_0,\,r)) \\
&&\quad\ls r^{2\az}\sum_{k\ge0}\bint_{B(x_0,\,r)} \bint_{B(x,\,2^{-k}r)}\int_{B(z,\,2^{-k+3}r)\setminus B(z,\,2^{-k}r)}\frac{|u(z)-u(w)|^2}{|z-w|^{n+2\az}}\,dw\,dz\,dx\\
&&\quad\ls r^{2\az}\sum_{k\ge0}\bint_{B(x_0,\,2r)} \bint_{B(z,\,2^{-k}r)}\int_{B(z,\,2^{-k+3}r)\setminus B(z,\,2^{-k}r)}\frac{|u(z)-u(w)|^2}{|z-w|^{n+2\az}}\,dw\,dx\,dz\\
&&\quad\ls r^{2\az}\sum_{k\ge0}  \bint_{B(x_0,\,2r)}\int_{B(z,\,2^{-k+3}r)\setminus B(z,\,2^{-k}r)}\frac{|u(z)-u(w)|^2}{|z-w|^{n+2\az}}\,dw \,dz\\
&&\quad\ls r^{2\az-n} \int_{B(x_0,\,2r)}\int_{B(z,\,8r)}\frac{|u(z)-u(w)|^2}{|z-w|^{n+2\az}}\,dw\,dz \\
&&\quad\ls \Phi_{\az}(u,\,B(x_0,\,16r)).
\end{eqnarray*}
This completes the proof of Proposition \ref{p3.1}.
\end{proof}

\section {Proof of Theorem \ref{t1.2}} \label{s4}

Here we only prove Theorem \ref{t1.2} under the assumption $\diam E<\fz$. The case $\diam E=\fz$ is similar.
Without loss of generality, we may assume that $\diam E=1$ and $E\subset B(0,\,1)$.
By Proposition \ref{p3.1}, it suffices to show that
$$
\Psi_{\az,\,2}(u\circ f,\,B)\ls\|u\|_{Q_\az(\rn)} \quad\hbox{for each ball}\ \ B=B(x_0,\,r).
$$
We divide the argument into two cases.

{\bf Case 1}:\ $d(x_0,\,E)\ge  4 r$.
Notice that $B(x ,2r)\cap E = \emptyset$ for all $x\in B(x_0,\,r)$. By $J_f\in A_1(\rn,\,E)$,
and $J_f(f^{-1}(z))J_{f^{-1}}(z)=1$ for almost all $z\in\rn$, for all $k\ge0$ and $x\in B(x_0,\,r)$,
we have
\begin{eqnarray}\label{e4.w1}\esup_{z\in f(B(x,\,2^{-k}r))}J_{f^{-1}}(z)&&=
  \esup_{z\in f( B(x,\,2^{-k}r)) }[J_{f}( f^{-1}(z) )]^{-1} \\
&&= \lf[ \einf _{w\in  B(x,\,2^{-k}r) } J_{f}( w )\r]^{-1}\nonumber\\
&&\ls  \frac{| B(x ,\,2^{-k}r ) |}{|f(B(x ,\,2^{-k}r ))|}.\nonumber
\end{eqnarray}
Thus,
\begin{eqnarray*}
&&  \bint_{  B(x ,\,2^{-k}r )}
|u\circ f(z)-c|^2\,dz\\
&&\quad=\frac{|f(B(x ,\,2^{-k}r ))|}{| B(x ,\,2^{-k}r ) |}
\bint_{f(B(x ,\,2^{-k}r ))}|u(z)-c|^2J_{f^{-1}}(z)\,dz\\
&&\quad \ls\bint_{f(B(x ,\,2^{-k}r ))}|u(z)-c|^2 dz.
\end{eqnarray*}
Hence we have
\begin{eqnarray*}
&& \Psi_{\az,\,2}(u\circ f,\,B(x_0,\,r))\\ 
&&\quad\ls \sum_{k\ge0} 2^{ 2k\az  } \bint_{ B(x_0,\,r) }\inf_{c\in\rr}
 \bint_{ f( B(x ,\,2^{-k}r ))} |u (z)-c|^2\,dz \,dx\\
 &&\quad\ls \sum_{k\ge0}   \int_{ B(x_0,\,r) }   \frac{{ | B(x ,\, r  )| ^{ 2\az/n-1}}}{|B(x,\,2^{-k}r)|^{2\az/n}}\inf_{c\in\rr}
 \bint_{ f( B(x ,\,2^{-k}r ))} |u (z)-c|^2\,dz \,dx.
 \end{eqnarray*}
Observe that $J_f\in A_1(\rn,\,E)$  also implies that
$$\frac{|f(B(x ,\,2^{-k}r))|}{|B(x,\,2^{-k}r)|}=\bint_{B(x,\,2^{-k}r)}J_f(z)\,dz\ls\einf_{z\in B(x,\,2^{-k}r)} J_f(z)\ls J_f(x)\quad \mbox{for almost all }\, x\in B,$$
that is,
$$|B(x,\,2^{-k}r)|^{-1}\ls J_f(x) |f(B(x ,\,2^{-k}r))|^{-1}.$$
Therefore, by this and a change of the variables again,
\begin{eqnarray*}
&& \Psi_{\az,\,2}(u\circ f,\,B(x_0,\,r))\\
&&\quad\ls \sum_{k\ge0}  \int_{B(x_0,\,r)}  \frac{ | B(x ,\, r  )| ^{ 2\az/n-1}}{|f(B(x,\,2^{-k}r))|^{2\az/n}} \inf_{c\in\rr}\bint_{ f( B(x ,\,2^{-k}r ))} |u (z)-c|^2\,dz\, [J_f(x)]^{2\az/n}\,dx\\
&&\quad\ls \sum_{k\ge0}  \int_{f(B(x_0,\,r))}   \frac{ | B(f^{-1}(x) ,\, r  )| ^{ 2\az/n-1}}{|f(B(f^{-1}(x),\,2^{-k}r))|^{2\az/n}} \\
&&\quad\quad\times \inf_{c\in\rr}\bint_{ f( B(f^{-1}(x) ,\,2^{-k}r ))} |u (z)-c|^2\,dz\, J_{f^{-1}}(x)[J_f(f^{-1}(x))]^{2\az/n}\,dx\\
&&\quad\ls \sum_{k\ge0}   \int_{f(B(x_0,\,r))}  \frac{ | B(f^{-1}(x) ,\, r  )| ^{ 2\az/n-1}}{|f(B(f^{-1}(x),\,2^{-k}r))|^{2\az/n}}\inf_{c\in\rr}\bint_{ f( B(f^{-1}(x) ,\,2^{-k}r ))} |u (z)-c|^2\,dz\, [J_{f^{-1}}(x)]^{1-2\az/n}\,dx.
\end{eqnarray*}
Now, by \eqref{e4.w1} with $k=0$ and $x=x_0$, we have
\begin{eqnarray*}\esup_{x\in f(B(x_0,\, r))}J_{f^{-1}}(x)
&&\ls  \frac{| B(x_0,\, r ) |}{|f(B(x_0,\, r ))|}\sim  \frac{| B(w,\, r ) |}{|f(B(w,\, r ))|}\quad \forall\ w\in B(x_0,\,r),
\end{eqnarray*}
which further yields   that
\begin{eqnarray}\label{e4.xx1}
&& \Psi_{\az,\,2}(u\circ f,\,B(x_0,\,r))\\
&&\quad\ls \sum_{k\ge0}   \bint_{f(B(x_0,\,r))}  \lf(\frac{|f(B(f^{-1}(x),\, r))|}{|f(B(f^{-1}(x),\,2^{-k}r))|}\r)^{2\az/n}
\inf_{c\in\rr}\bint_{ f( B(f^{-1}(x) ,\,2^{-k}r ))} |u (z)-c|^2\,dz\,  \,dx\nonumber\\
&&\quad\ls \bint_{f(B(x_0,\,r))}\sum_{k\ge0} \left(\frac {L_f (f^{-1}(x),r)} {L_f(f^{-1}(x),\,2^{-k}r)}\right)^{2\az}
 \inf_{c\in\rr}\bint_{ f( B(f^{-1}(x) ,\,2^{-k}r ))}|u (z)-c|^2\,dz \,dx,\nonumber
\end{eqnarray}
where
$$
L_f(z,r)=\sup \{|f(z)-f(w)|:|z-w|\le r\}\ \ \&\ \ L_f(z,\,r)^n\sim |f(B(z,\,r))|.
$$
Moreover, by quasisymmetry of $f$,  for all $j\in\zz$ and $z\in\rn$, we have
\begin{equation}\label{e4.1}
 \sharp\lf\{k\in\zz:\  L_f\lf(z,\,2^{-k}r\r)\in[2^{-j-1}L_f(z,\, r),\,2^{-j}L_f(z,\, r)) \r\}\ls 1.
\end{equation}

Recalling that
$$
f(B(x_0,\,r))\subset B(f(x_0),\,L_f(x_0,\,r))\ \ \&\ \ L_f(f^{-1}(x),\,r)\le 2^{N_2}L_f(x_0,\,r)
$$
holds for some constant $N_2\ge 1$ (independent of $x_0,\,r$; see \cite{k11}), we arrive at
\begin{eqnarray}\label{e4.2}
&& \Psi_{\az,\,2}(u\circ f,\,B(x_0,\,r))\\
&&\quad\ls \sum_{j\ge0} 2^{2j\az} \bint_{B(f(x_0),\,L_f(x_0,\, r))}
 \bint_{ B(x ,\,2^{-j}2^{N_2}L_f(x_0,\,r))}|u(z)-c|^2\,dz \,dx\nonumber\\
&&\quad\ls \Psi_{\az,\,2}(u ,\,B(f(x_0),\,2^{N_2}L_f(x_0,\, r))), \nonumber
\end{eqnarray}
which together with Proposition \ref{p3.1} gives
$$\Psi_{\az,\,2}(u\circ f,\,B(x_0,\,r))\ls\|u\|_{Q_\az(\rn)}^2 $$
as desired.

{\bf Case 2}:\ $ d(x_0,\,E)<4r\le 4$.
Recall that each domain $\boz$ admits a Whitney decomposition. In particular, for   $\boz=\rn\setminus E$, there exists
 a collection $W_{\boz}=\{S_j\}_{j\in\nn}$ of countably many dyadic (closed) cubes such that

 (i) $\boz=\cup_{j\in\nn}S_j$ and $ (S_k)^\circ\cap (S_j)^\circ=\emptyset$ for all $j,\,k\in\nn$  with $j\ne k$;

 (ii) $2^7\sqrt n\ell(S_j)\le \dist(S_j,\,\partial\boz)\le 2^9\sqrt n \ell(S_j)$;

 (iii) $\frac14\ell(S_k)\le  \ell(S_j) \le 4\ell(S_k)$ whenever  $S_k\cap S_j\ne\emptyset$.

Assume that $2^{-k_0-1}\le 16r<2^{-k_0}$ for $k\in\nn$.
For each $k\in\zz$, write
$$
\mathscr S_k(16B)=\{S_j\in W_{\boz}:\ S_j\cap  16B \ne \emptyset,\,2^{-k}\le \ell(S_j)<2^{-k+1}\}\equiv\{S_{k,i}\}_i.
$$
Notice that there exists a integer $N_0$ such that if $k\le k_0-N_0$, then $\mathscr S_k(16B)=\emptyset$. Indeed,
since
$$
\dist(S_{k,j},\,E)\le 16r+\dist(x_0,\,E)\le 20r,
$$
by (ii) above, we have   $2^{-k}\ls 2^{-k_0}$ which is as desired.
Moreover, letting $\ez\in(0,\,n-\dim_L E-2\az)$, we claim that
for all $k\ge k_0-N$,
$$\sharp\mathscr S_k(16B) \ls 2^{(k-k_0)(\overline \dim_LE+\ez)}.$$
To see this, by the definition of $\overline\dim_L E $
there exists  constants $N_1\ge8$ and $k_1\in\nn$ such that for all $k\ge k_1+k_0+N_1$, we have
 $$ \frac{\log N_{\rm cov}(2^{-k},\,E\cap 32B)}{ \log( 32r /2^{-k} )}\le \overline\dim_L E+\ez ,$$
which implies that
\begin{equation}\label{e4.3}
 N_{\rm cov}( 2^{-k}, E\cap 32B)\ls 2^{(k-k_0)(\overline\dim_LE+\ez)}.
\end{equation}
For every $\dz>0$, denote by $\mathscr N_{\rm cov}(\dz,\,E\cap 32B)$ the collection of cubes of edge length $\dz$
 required to cover $E\cap 32B$ and
 $$
 \sharp \mathscr N_{\rm cov}(\dz,\,E\cap 32B)=N_{\rm cov}(\dz,\,E\cap 32B).
 $$
For $k\ge -N_0$ and $S_{k,i}\in \mathscr S_k(16B)$, we have
$2^{11}\sqrt n S_{k,i}\cap E\ne\emptyset$ and hence $S_{k,i}$ intersects  some cube
$Q\in \mathscr N_{\rm cov}( 2^{-k},\,E\cap 32B)$, which implies that $S_{k,i}\subset 2^{13} nQ$.
Also notice that for each cube $Q\in \mathscr N_{\rm cov}( 2^{-k},16B\cap E)$, the cube $2^{13} nQ$ can only contain  a  uniformly bounded number of
$S_{k,i}\in \mathscr S_k(16B)$. We conclude that for $k\ge  -N_0$,
$$\sharp \mathscr S_a(16B )\ls   N_{\rm cov}(2^{-k}, E\cap 32B).$$
This together with \eqref{e4.3} gives that
for $k\ge k_1+k_0+N_1$,
$$\sharp \mathscr S_k (16B)\ls  2^{(k-k_0)(\overline\dim_LE+\ez)}.$$
On the other hand, if $k_0-N_0\le k\le k_1+k_0+N_1$, then by $2^{k-k_0}\le 2^{k_1+N_1+N_0}\ls 1$ we always have
$$\sharp \mathscr S_k (16B)\ls 2^{n(k-k_0)}\ls 2^{(k-k_0)(\overline\dim_LE+\ez)}.$$
This gives the above claim.

By Proposition \ref{p3.1}, we have
\begin{eqnarray*}
 &&\Psi_{\az,\,2}(u\circ f,\,B(x_0,\, r))\\
&&\quad\ls \Phi_{\az}(u\circ f,\,B(x_0,\,16r))\\
&&\quad \ls r^{2\az-n} \sum_{k\ge k_0-N_0} \sum_{i=1}^{\sharp \mathscr S_k(16B )}  \int_{S_{k,i}} \int_{B(x_0,\,16r)} \frac {|u\circ f(x)-u\circ f(y)|^2} {|x-y|^{n+2\az}}\,dx\, dy \\
&&\quad\ls r^{2\az-n} \sum_{k\ge k_0-N_0} \sum_{i=1}^{\sharp \mathscr S_k(16B )}    \int_{S_{k,i}} \int_{ 2S_{k,i}} \frac {|u\circ f(x)-u\circ f(y)|^2} {|x-y|^{n+2\az}}\, dx\, dy\\
&&\quad\quad +r^{2\az-n} \sum_{k\ge k_0-N_0} \sum_{i=1}^{\sharp \mathscr S_k(16B )}   \int_{S_{k,i}} \int_{B(x_0,\,16r)\setminus 2S_{k,i}} \frac {|u\circ f(x)-u\circ f(y)|^2} {|x-y|^{n+2\az}}\, dx\, dy\\
&&\quad=P_1 + P_2.
\end{eqnarray*}

For each $S_{k,i}$, let  $B_{k,i}$ be the ball centered at $x_{k,i}$ ($x_{k,\,i}$ is the center of $S_{k,i}$) and radius $2\sqrt n \ell(S_{k,i})$.
Then
$$
2S_{k,i}\subset B_{k,i}\ \ \&\ \ \dist(x_{k,i},E)\ge 4\cdot  16\cdot2\sqrt n\ell (S_{k,i}).
$$
So applying the above {\bf Case 1} to $16B_{k,i}$, we have
$$\Phi_{\az}(u\circ f,\,B_{k,i})\ls \Psi_{\az,\,2}(u\circ f,\,16B_{k,i})\ls \|u\|^2_{Q_\az(\rn)}.$$
This, together with $n-2\az-\overline\dim_LE-\ez>0$, gives
\begin{eqnarray*}
P_1&&\ls  r^{2\az-n} \sum_{k\ge k_0-N_0} \sum_{i=1}^{\sharp \mathscr S_a(16B )}   |B_{k,i}|^{1-2\az/n} \Phi_{\az}(u\circ f,\,B_{k,i})\\
&&\ls  r^{2\az-n} \sum_{k\ge k_0-N_0} \sum_{i=1}^{\sharp \mathscr S_a(16B )}  2^{-(n-2\az)k} \|u\|^2_{Q_\az(\rn)}\\
&&\ls   \sum_{k\ge k_0-N_0}  2^{(k-k_0)(\overline\dim_LE+\ez)} 2^{(n-2\az)(k_0-k)} \|u\|^2_{Q_\az(\rn)}\\
&&\ls    \|u\|^2_{Q_\az(\rn)}.
\end{eqnarray*}

To estimate $P_2$,   write
\begin{eqnarray*}
 &&\int_{S_{k,i}} \int_{B(x_0,\,16r)\setminus 2S_{k,i}} \frac {|u\circ f(x)-u\circ f(y)|^2} {|x-y|^{n+2\az}}\, dx\, dy\\
&&\quad\ls\sum_{\ell= 1}^{k-k_0+5} 2^{(\ell-k)(-n-2\az)}\int_{S_{k,i}}\int_{2^{\ell+1}S_{k,i}\setminus 2^\ell S_{k,i}} |u\circ f(x)-u\circ f(y)|^2 \, dx\, dy\\
&&\quad\ls \sum_{\ell= 1}^{k-k_0+5} 2^{-2\az(\ell-k)}2^{-kn} \lf\{  \bint_{S_{k,i}} |u\circ f(x)-(u\circ f)_{2^{\ell+1}S_{k,i}}|^2\,dx \r. \\
&&\quad\quad\lf.+  \bint_{2^{\ell+1} S_{k,i} } |u\circ f(y)-(u\circ f)_{2^{\ell+1}S_{k,i}}|^2  \,  dy\r\}.
  \end{eqnarray*}
Observing that
\begin{eqnarray*}
&& \lf\{\bint_{S_{k,i}} |u\circ f(x)-(u\circ f)_{2^{\ell+1}S_{k,i}}|^2\,dx\r\}^{1/2}\\
&&\quad\ls \sum_{j=1}^{\ell+1}
 \lf\{\bint_{2^jS_{k,i}} |u\circ f(x)-(u\circ f)_{2^{j}S_{k,i}}|^2\,dx\r\}^{1/2}\\
&&\quad\ls (\ell+1)
\|u\circ f\|_{BMO(\rn)},
\end{eqnarray*}
we obtain
\begin{eqnarray*}
 &&\int_{S_{k,i}} \int_{B(x_0,\,16r)\setminus 2S_{k,i}} \frac {|u\circ f(x)-u\circ f(y)|^2} {|x-y|^{n+2\az}}\, dx\, dy\\
&&\quad\ls \sum_{\ell= 1}^{k-k_0+5} 2^{-2\az(\ell-k)} 2^{-kn} (\ell+1)^2 \|u\circ f\|^2_{BMO(\rn)}\\
&&\quad\ls 2^{(2\az-n)k} \|u\circ f\|^2_{BMO(\rn)}.
  \end{eqnarray*}
Therefore, by $n-2\az -\overline\dim_LE-\ez>0$, one gets
\begin{eqnarray*}
P_2&&\ls r^{2\az-n} \sum_{k\ge k_0-N_0} \sum_{i=1}^{\sharp \mathscr S_a(16B )} 2^{(2\az-n)k  )} \|u\circ f\|^2_{BMO(\rn)}\\
&& \ls \sum_{k\ge k_0-N_0}  2^{(k-k_0)(\overline\dim_LE+\ez)} 2^{(n-2\az)(k_0-k)} \|u\circ f\|^2_{BMO(\rn)}\\
&& \ls \|u\circ f\|^2_{BMO(\rn)}.
\end{eqnarray*}
Recall that it was proved by Reimann \cite{r74} that
 $\|u\circ f\|_{BMO(\rn)}\ls \|u\|_{BMO(\rn)} $,
and also in \cite{ejpx} that
 $\|u\|_{BMO(\rn)}\ls \|u\| _{Q_\az(\rn)}.$
Thus $ P_2 \ls \|u \|^2_{Q_\az(\rn)}.$

Combining the estimates for $P_1$ and $P_2$, we arrive at
$\Psi_{\az,\,2}(u\circ f,\,B(x_0,\, r))\ls \|u \|^2_{Q_\az(\rn)}$
for all $x_0$ and $r$ as desired.

{\bf Case 3}:\ $d(x_0,\,E)\le 2r$ and $r>1$. Without loss of generality, we may assume that $x_0=0$.
Denote by $M$ the minimum number of balls, which are centered in $B(0,\,1)\setminus B(0,\,1/2)$ and have radius $2^{-9}$,
required to cover $B(0,\,1)\setminus B(0,\,1/2)$. Let $\{B_j\}_{j=1}^M$ be a sequence of such balls and write their  centers as $\{x_j\}_{j=1}^M$.
Write
$$
B_{k,\,j}=  B( 2^{k }x_{j},\,2^{k-9} )\ \ \hbox{for}\ \ k\ge 2\ \ \hbox{and}\ \ j=1,\,\cdots,\,M.
$$
Notice that
\begin{equation}\label{e4.xx2}2^{k-9}= 2^{k-2} 2^{-7} \le 2^{-7} d(2^{k }x_{j},\,E) .
\end{equation}
Assume that $2^{ k_0-1 }\le r<2^{ k_0}$.
Then $k_0\ge 1$, and  $B(x_0,\,16r)\setminus B(x_0,\,2)$ can be covered by the
family   $\{B_{k,\,j}:\ 2\le k\le k_0+4,\,0\le j\le M\}$. Write $B_{1,\,j}=B(0,\,2)$.
Then we have
\begin{eqnarray*}
&&\Psi_{\az,\,2}(u\circ f,\,B(x_0,\, r))\\
&&\quad\ls \Phi_{\az}(u\circ f,\,B(x_0,\,16r))\\
&&\quad \ls \sum_{k=1}^{k_0+4}\sum_{j=1}^M r^{2\az-n} \int_{B_{k,\,j}}\int_{16B }\frac{|u\circ f(x)-u\circ f(y)|^2}{|x-y|^{n+2\az}}\,dx\,dy\\
&&\quad \ls\sum_{k=1}^{k_0+4} r^{2\az-n} 2^{k(n-2\az)} \Phi_{\az}(u\circ f,\,2B_{k,\,j})\\
&&\quad\quad+
 \sum_{k=1}^{k_0+4} r^{2\az-n} \int_{B_{k,\,j}}\int_{16B\setminus 2B_{k,\,j}}\frac{|u\circ f(x)-u\circ f(y)|^2}{|x-y|^{n+2\az}}\,dx\,dy\\
&&\quad= P_3+P_4.
\end{eqnarray*}
By Proposition \ref{p3.1} and  the result of {\bf Case 1} applied to $32  B_{k,j}$, we have $$\Phi_{\az}(u\circ f,\,2B_{k,\,j})\ls\Psi_{\az,\,2}(u\circ f,\,32B_{k,\,j})\ls \|u\|^2_{Q_\az(\rn)}$$ where
$$32\cdot 2^{k-9}=32\cdot 2^{-7}d(2^{k }x_{j},\,E) \le d(2^{k }x_{j},\,E)/4$$ due to \eqref{e4.xx2},
and hence
$$P_3\ls \|u\|^2_{Q_\az(\rn)} \sum_{k=1}^{k_0+4} r^{2\az-n} 2^{k(n-2\az)}\ls \|u\|^2_{Q_\az(\rn)}.$$
For $P_4$, an argument similar to $P_2$ in the {\bf Case 2} leads to $P_4\ls \|u\|^2 _{Q_\az(\rn)}$.
This finishes the proof of Theorem \ref{t1.2}.

\section {Proofs of Corollaries \ref{c1.3} and \ref{c1.x3}} \label{s4c}

\begin{proof}[Proof of Corollary \ref{c1.3}]
Notice that if $\bz>0$, then $f$ is a quasiconformal mapping from $\rn\to\rn$,
and that
$$
\begin{cases}
J_f\in A_1(\rn,\{0\})\ \ \hbox{when}\ \ \bz>1;\\
J_f\in A_1(\rn)\ \ \hbox{when}\ \ 0<\bz<1.
\end{cases}
$$
By Theorem \ref{t1.2}, if $\bz>0$, then ${\bf C}_f$ is bounded on $Q_\az(\rn)$ for all $\az\in(0,\,1)$.
If $\bz<0$, then $f$ is not a quasiconformal mapping from $\rn\to\rn$; so we can not apply Theorem \ref{t1.2} directly.
However, observe that $f$ is a quasiconformal mapping from $\rn\setminus\{0\}$ to $\rn$ with $J_f(x)\sim |x|^{\bz-1}$ yielding $J_f\in A_1(\rn, \{0\})$.
Thus, an argument similar to but easier than that for Theorem \ref{t1.2}  will lead to the boundedness of ${\bf C}_f$ on $Q_\az(\rn)$ for all $\az\in(0,\,1)$.

Indeed, let $u\in Q_\az(\rn)$ and $B=B(x_0,\,r)$ be an arbitrary ball of $\rn$. If $r<|x_0|/4$, then
$$
J_f(x)\sim |x_0|^{\bz-1}\quad\forall\quad x\in B(x_0,\,3r).
$$
With the help of this and $J_f\in A_1(\rn,\{0\})$, similarly to {\bf Case 1}  in the proof of Theorem \ref{t1.2}, we obtain
\eqref{e4.xx1}, \eqref{e4.1} and \eqref{e4.2}. This implies
$$
\Psi_{\az,\,2}(u\circ f,\,B(x_0,\,r))\ls\|u\|_{Q_\az(\rn)}^2.
$$

If $r\ge |x_0|/4$,  then by $B(x_0,\,r)\subset B(0,\,2r)$ and Proposition \ref{p3.1}, we have
  $$\Psi_{\az,\,2}(u\circ f,\,B(x_0,\,r))\ls\Psi_{\az,\,2}(u\circ f,\,B(0,\,2r))\ls \Phi_\az(u\circ f,\,B(0,\,32r)).$$
  Similarly to {\bf Case 3}  in the proof of Theorem \ref{t1.2},
denote by $M$ the minimum number of balls (centered at $B(0,\,1)\setminus B(0,\,1/2)$ and having radii $2^{-9}$),
that are required to cover $B(0,\,1)\setminus B(0,\,1/2)$. Let $\{B_j\}_{j=1}^M$ be   a collection of such balls and write their centers as $\{x_j\}_{j=1}^M$.
Write
$$
B_{k,\,j}=  B( 2^{-k }2^5rx_{j},\,2^{-k-9}2^5r)\ \ \hbox{for}\ \ k\ge 0\ \  \hbox{and}\ \ j=1,\,\cdots,\,M.
$$
Then $B( 0,\,32r)\setminus \{0\}$ is covered by the
family of balls $\{B_{k,\,j}:\ k\ge0,\,0\le j\le M\}$.
Therefore, we obtain
\begin{eqnarray*}
&&\Psi_{\az,\,2}(u\circ f,\,B(x_0,\, r))\\
&&\quad \ls \sum_{k\ge0} \sum_{j=1}^M r^{2\az-n} \int_{B_{k,\,j}}\int_{ B(0,\,32r)}\frac{|u\circ f(x)-u\circ f(y)|^2}{|x-y|^{n+2\az}}\,dx\,dy\\
&&\quad \ls\sum_{k\ge0}  r^{2\az-n} (2^ {-k}r)^{ n-2\az } \Phi_{\az}(u\circ f,\,2B_{k,\,j})\\
&&\quad\quad+
 \sum_{k\ge0} r^{2\az-n} \int_{B_{k,\,j}}\int_{B(0,\,32r)\setminus 2B_{k,\,j}}\frac{|u\circ f(x)-u\circ f(y)|^2}{|x-y|^{n+2\az}}\,dx\,dy\\
&&\quad= P_5+P_6.
\end{eqnarray*}
Similarly to the estimate on $P_3$, we have
$
P_5\ls \|u\|_{Q_\az(\rn)}^2;
$
and similarly to but easier than for $P_2$, we obtain
$
P_6\ls \|u\|_{Q_\az(\rn)}^2.
$
Putting all together gives
$$
\Psi_{\az,\,2}(u\circ f,\,B(x_0,\,r))\ls\|u\|_{Q_\az(\rn)}^2,
$$
as desired, and hence finishes the proof of Corollary \ref{c1.3}.
\end{proof}

\begin{proof}[Proof of Corollary \ref{c1.x3}] For our convenience, let $\rr^n_+=\{z=(x,\,y):\ x\in\rr^{n-1}\ \& \ y>0 \}$.
We also write $ \hh^n=\rr^n_+\setminus\rr^{n-1}$ and equip it with the hyperbolic distance $d_{\hh^{n}}$ - that is -
$$d_{\hh^n}(w,\,w')=\inf_\gz\int_\gz\frac{|dz|}{y}\quad\forall\quad w,\, w'\in \hh^n,
$$
where the infimum is taken over all rectifiable curves $\gz$ in $\hh^n$ joining $w$ and $w'$.

Suppose that $g:\ \rr^{n-1}\to\rr^{n-1}$ is a quasiconformal mapping when $n\ge 3$, or a quasisymmetric mapping when $n=2$.
According to Tukia-V\"ais\"al\"a \cite[Theorem 3.11]{tv}, $g$ can be extended to  such a quasiconformal mapping $f:\rr^n_+\to\rr^n_+$ that

(i) $f|_{\rr^{n-1}}=g$;

(ii)  $f|_{\hh^n}$ is an $L$-biLipschitz with respect to $d_{\hh^n}$ for some constant $L\ge 1$, i.e.,

$$
\frac 1L d_{\hh^n}(z,\,w)\le d_{\hh^n}(f(z),\,f(w))\le L d_{\hh^n}(z,\,w)\quad\forall\quad w,\,w'\in\hh^n.
$$

\noindent Obviously, such an $f$ can be further extended to a quasiconformal mapping  $\wz f: \ \rn\to\rn$ by reflection, that is,
 $$
 \wz f(z)=
 \begin{cases}
 f(z_1,\,\cdots,\,z_{n-1},\,-z_n)\quad\hbox{for}\quad z\in\rn\setminus\rr^n_+;\\
 f(z)\quad\hbox{for}\quad z\in\rr^n_+.
 \end{cases}
 $$
For sake of simplicity, we   write $\wz f$ as $f$, and generally set
 $$
 \begin{cases}
 n\ge 3;\\
 2\le p<n;\\
 \hh^{n,\,p}=\rr^n\setminus \rr^p=\{z=(x ,y):\ x \in\rr^{n-p}\ \& \ 0\ne y\in\rr^p \}.
 \end{cases}
 $$
We equip $ \hh^{n,\,p}$ with the distance $d_{\hh^{n,\,p}}$, an analog of the hyperbolic distance,  via
$$
d_{\hh^{n,p}}(w,\,w')=\inf_\gz\int_\gz\frac{|dz|}{|(0,\,y)|  }\quad\forall\quad w,\,w'\in \hh^{n,\,p},
$$
where the infimum is taken over all rectifiable curves $\gz$ in $\hh^{n,\,p}$ joining $w$ and $w'$. Suppose that $g:\ \rr^{n-p}\to\rr^{n-p}$ is a quasiconformal mapping  when $n-p\ge 2$, or a quasisymmetric mapping when $n-p=1$. In accordance with
Tukia-V\"ais\"al\"a's \cite[Section 3.13]{tv}, $g$ can be extended to a quasiconformal mapping $f:\rr^n \to\rr^n $ such that

(i) $f|_{\rr^{n-p}}=g$;

(ii)  $f|_{\hh^{n,\,p}}$ is a $L$-biLipschitz with respect to $d_{\hh^{n,\,p}}$ for some constant $L\ge 1$.

Notice that both $f$ and $f^{-1}$ are biLipschitz  with respect to $d_{\hh^{n,\,p}}$.
We show that ${\bf C}_f$ is bounded; the case of ${\bf C}_{f^{-1}}$ is analogous.
By Theorem \ref{t1.2}, it suffices to verify $J_f \in A_1(\rn;\rr^{n-p})$.
In what follows, we only consider the case $p=1$; the argument can easily be modified   to handle the case $p\ge 2$.

First observe that
$$
J_f(z )\sim  \frac {[d(f (z ),\,\rr^{n-1})]^n} {|y| ^n} \quad a.e.\quad z =(x,\,y)\in\rn\setminus\rr^{n-1},
$$
where $d(f(z),\,\rr^{n-1})$ stands for the Euclidean distance from the point $f(z)$ to $\rr^{n-1}$. Indeed, upon taking $r>0$ small enough such that
$$
r<|y |/2\quad\&\quad L_{f }(z ,\,r)\le d(f (z ),\,\rr^{n-1})/2,
$$
we get
$$
d(w,\,\rr^{n-1})\sim d(z ,\,\rr^{n-1})\sim |y |\quad \&\quad d(f(w),\,\rr^{n-1})/2\sim d(f (z ),\,\rr^{n-1})/2\quad \forall\ w\in B(z ,\,r),
$$
which in turn implies
$$
d_{\hh^n}(z,\,w) \sim  \frac{|z-w|}{|y|} \quad \&\quad d_{\hh^n}(f(z),\,f(w)) \sim \frac{|f(z)-f(w)|}{d(f (z ),\,\rr^{n-1})} \quad \forall\quad w\in B(z ,\,r).
$$
Therefore
$$
J_f(z)\sim |Df(z)|^n\sim \frac{[d(f (z ),\,\rr^{n-1})]^n} {|y| ^n} \quad a.e.\quad z\in \rn,
$$
as desired.

Now let $B(x_0,\,r)$ be an arbitrary ball with radius $r\le |y_0|/2$ and $z_0=(x_0,\,y_0)$.
Obviously, we have
$$
|y|/2\le |y_0|\le 2|y|\quad\forall\quad z=(x,\,y)\in B(z_0,\,r).
$$
 Then, it is enough to prove that
 \begin{equation}\label{e4.xxx9}
 d(f (z_0 ),\,\rr^{n-1})\sim d(f (z  ),\,\rr^{n-1})\quad a.e.\quad z \in B(z_0,\,r).
 \end{equation}
Assuming this holds for the moment, we have
$$
J_f(z )\sim   \frac{[d(f (z_0 ),\,\rr^{n-1})]^n} {|y_0| ^n} \quad a.e.\quad z \in B(z_0,\,r)
$$
and further
$$
\bint_{B(x_0,\,r)}J_f(z)\,dz \sim   \frac {[d(f(z_0 ),\,\rr^{n-1})]^n}{|y_0| ^n}\sim      \einf _{z\in B(x_0,\,r)} J_f(z),
$$
that is, $J_f\in A_1(\rn;\rr^{n-1})$, as desired.

Towards \eqref{e4.xxx9}, note that $f$ is a quasisymmetric mapping. So, there exists a homeomorphism $\eta:[0,\,\fz)\to [0,\,\fz)$ such that
$$
\frac{|f(z)-f(w)|}{|f(z_0)-f(w)|}\ls  \eta\lf(\frac{|z-w|}{|z_0-w|}\r)\quad \forall\quad w\in\rn.
$$
Observe that
$$
\frac12|z_0-w|\le |z_0-w|-|z-z_0|\le |z-w|\le |z-z_0|+|z_0-w|\le 2|z_0-w| \quad \forall\quad w\in\rr^{n-1}.
$$
Thus, by taking such a point $w\in\rr^{n-1}$ that
$$
|f(z_0)-f(w)|=d(f(z_0),\,\rr^{n-1}),
$$
we have
$$
d(f(z) ,\,\rr^{n-1})\le |f(z)-f(w)| \le  \eta(2)|f(z_0)-f(w)|\ls d(f(z_0),\,\rr^{n-1}) .
$$
Upon changing the roles of $z$ and $z_0$, we also have
$$
d(f(z_0) ,\,\rr^{n-1})\ls d(f(z),\,\rr^{n-1}).
$$
Hence \eqref{e4.xxx9} holds. This completes the proof of Corollary \ref{c1.x3}.
\end{proof}

\section{Proofs of Theorems \ref{t1.4} and  \ref{t5.1}} \label {s5}

\begin{proof}[Proof of Theorem \ref{t1.4}]

Fix $\az_0\in(0,\,1)$.  Let $a=1-2^{-2\az_0/(n-2\az_0)}\in (0,\, 1 )$, and let the sets $E_a$    be as  \eqref{e2.y2}  in Section 2.
Then we have
$
n-2\az_0 =n/ \log[2/(1-a)]
$
and by Lemma \ref{lemma 2.4},
$
\overline\dim_L E_a = n-2\az_0.
$
The set  $E_a$    is exactly what we want in the statement of Theorem \ref{t1.4}.

Now we are going to construct
a quasiconformal (Lipschitz) mapping $f:\rn\to\rn$ such that
 $J_f\in A_1(\rn,\,E_a)$ and hence  $J_f\in A_1(\rn,\,\mathcal E_a)$  but ${\bf C}_f$ is unbounded on $Q_{\az}(\rn)$
 for any $\az \in(\az_0 ,1)$.

Recall that  $\{z_{m,j}\}$ are the centers of $\{Q_{m,j}\}$ and $\{Q_{m,j}\}$ are
the pre-cubes appearing in the Cantor construction $ E_a$, see Section 2.
Let $\bz\in( 0,\,\fz)$ and define the map $f$ by setting
$$
f(x)=  \lf(\frac12 a[(1-a)/2]^{m}\r)^{-\bz} |x-z_{m,\,j}|^\bz(x-z_{m,\,j})+z_{m,\,j}
$$
if
$$
|x-z_{m,\,j}|<\frac12 a[(1-a)/2]^{m}\ \ \hbox{for\ some}\ \ m\in\nn\ \ \hbox{and}\ \ j=1,\,\cdots,2^{mn},
$$
and $f(x)=x$ otherwise.  Indeed, we only perturb the identity mapping  on all balls
$$
B(z_{m,\,j},\, \frac12 a[(1-a)/2]^{m})\subset Q_{m,\,j}
$$
by making ``radial" stretchings
with respect to their centers, where $|Q_{m,j}|=a[(1-a)/2]^{mn}$.
Notice that
$$
J_f(x)\sim|Df(x)|^n\sim \lf(\frac12 a[(1-a)/2]^{m}\r)^{-n\bz } |x-z_{m,\,j}|^{n\bz}\ls1\ \ \hbox{when}\ \ |x-z_{m,\,j}|<\frac12 a[(1-a)/2]^{m},
$$
and
$$
J_f(x)=|Df(x)|^n=1\ \ \hbox{otherwise}.
$$
Thus $f$ is  a quasiconformal mapping.
  Moreover, it is easy to check that
  $$
  J_f\in A_1(\rn;\,E_a)\ \ \&\ \
  J_f\notin A_1(\rn).
  $$

Set
$$ \bz_0=  1+\frac{n-2\az}{ n}\log\lf(\frac{1-a}2\r).$$
Then
$$
\bz_0>0\ \ \hbox{since}\ \ n-2\az<n-2\az_0 =\frac n{ \log[2/(1-a)]}.
$$
Set also
$$
\begin{cases}
\ell= mn\bz/(n-2\az)\ \ \hbox{if}\ \ 0<\bz\le \bz_0;\\
\ell= mn\bz_0/(n-2\az)\ \ \hbox{if}\ \ \bz>\bz_0.
\end{cases}
$$

With each $z_{m,\,j}\in E$, we   associate   a ball $B_{m,\,j}$ such that
$$
B_{m,\,j}\subset \frac {17}{64} 2^{-\ell}a  Q_{m,j}\ \ \&\ \
r_{m,\,j}=\frac1{64} 2^{-\ell} a[ (1-a)/2]^{ m}
$$
and so that the center $ x_{m,\,j}$ of $B_{m,\,j}$ satisfies
$$
|x_ {m,\,j}-z_{m,\,j}|=\frac14 2^{-\ell}a[(1-a)/2]^{m}.
$$
For each $m$, set
$$
u_m=\sum_{j=1}^{2^{mn}}u_{m,j},
$$
where
$$
u_{ m,\,j }(x)=\chi_{B_{m,\,j}} d(x,\,\partial B_{m,j})\ \ \hbox{for\ all\ possible}\ \ j.
$$
Obviously, $u_{m,j}$ is a Lipschitz function.

We make two claims:
\begin{equation}\label{e5.1} \|u_m\|^2_{Q_\az(\rn)} \ls  2^{mn} 2^{-\ell(n+2-2\az)} [(1-a)/2]^{m(n+2-2\az)}
\end{equation}
  and
\begin{equation}\label{e5.2}
\|u_m\circ f\|^2_{Q_\az(\rn)}
 \gs 2^{mn} 2^{  - \ell (n-2\az) /(\bz+1)} 2^{- 2\ell } [ (1-a)/2]^{m(n-2\az+2) } .
\end{equation}

Assuming that both \eqref{e5.1} and \eqref{e5.2} hold for the moment, we arrive at
$$
\frac{\|u_m\circ f\|^2_{Q_{\az}(\rn)}} {\|u_m\|^2_{Q_{\az}(\rn)}}
\gs\frac{2^{mn} 2^{  - \ell (n-2\az) /(\bz+1)} 2^{- 2\ell } [ (1-a)/2]^{m(n-2\az+2)}}
{2^{mn} 2^{-\ell(n+2-2\az) } [ (1-a)/2]^{m(n+2-2\az)}}\gs  2^{   \ell (n-2\az) \bz/(\bz+1)} ,
$$
which tends to $\fz$ as $m\to\fz$ since $\bz>0$ and $\ell \sim m$. This gives Theorem \ref{t1.4} under \eqref{e5.1}-\eqref{e5.2}.

Finally, we  verify \eqref{e5.1}-\eqref{e5.2}.

{\bf Proof of \eqref{e5.1}.} Let $B=B(x_B,\,r_B)$ be an arbitrary ball.

If $r_B\le r_{m,j}$,
since
$$
|u_m(x)-u_m(y)|\le |x-y|\ \ \forall\ \ x,\,y\in\rn
$$
one has
\begin{eqnarray}\label{e5.3}
\Phi_\az(u_m,\,2B )
 && \ls r_B^{2\az-n} \int_{2B } \int_{2B } \frac {1} {|x-y|^{n-2(1-\az)}} dx\, dy\\
&& \ls r_B^{2\az-n} \int_{2B } \int_{B (y,\,2r_B)} \frac {1} {|x-y|^{n-2(1-\az)}} dx\, dy\nonumber\\
 && \ls r_B^{2}\ls r_{m,\,j}^{2}.\nonumber
\end{eqnarray}
In particular,
$
\Phi_\az(u_m,\,2B_{m,j} )\ls r_{m,\,j}^{2}.
$

If $r_B>r_{m,j}$, one writes
 \begin{eqnarray*}
\Phi_\az(u_m,\,2B)&&\le 2|B|^{2\az/n-1}\sum_{B_{m,\,j}\cap 2B\ne\emptyset}\int_{B_{m,j}}\int_{2B} \frac {|u_m(x)-u_m(y)|^2} {|x-y|^{n+2\az}} dx\, dy
 \\
&&\le |B|^{2\az/n-1}\sum_{B_{m,\,j}\cap 2B\ne\emptyset}|B_{m,j}|^{1-2\az/n} \Phi_\az(u_m,\,2B_{m,j}) \\
&&\quad+ |B|^{2\az/n-1}\sum_{B_{m,\,j}\cap 2B\ne\emptyset} \int_{B_{m,j}}\int_{2B\setminus 2B_{m,j}} \frac {|u_m(x)-u_m(y)|^2} {|x-y|^{n+2\az}} dx\, dy.
\end{eqnarray*}
Notice that
\begin{eqnarray*}\int_{B_{m,j}}\int_{2B\setminus 2B_{m,j}} \frac {|u_m(x)-u_m(y)|^2} {|x-y|^{n+2\az}} dx\, dy&&\ls
  r_{m,j}^2|B_{m,j}| \int_{2B\setminus 2B_{m,j}} \frac 1 {|y-z_{m,\,j}|^{n+2\az}}  dy\\
&&\ls   r_{m,j}^{2-2\az}|B_{m,j}|.
\end{eqnarray*}
So, by \eqref{e5.3} one has
\begin{equation}\label{e5.4}
\Phi_\az(u_m,\,2B)
 \ls |B|^{2\az/n-1}\sum_{B_{m,\,j}\cap 2B\ne\emptyset} r_{m,j}^{2-2\az+n}.
\end{equation}
Below we consider three subcases.

First, if $r_{m,j}< r_B\le \frac1{64}  a[ (1-a)/2]^{m }$,
there are a uniformly bounded number of balls $B_{m,\,j}$ such that
$B_{m,j}\cap 2B\ne\emptyset$ and hence
\begin{eqnarray*}
\Phi_\az(u_m,\,2B)
 &&\ls |B|^{2\az/n-1}   \{2^{- \ell } [ (1-a)/2]^{ m}\} ^{2-2\az+n}\ls 2^{-{2}\ell } [ (1-a)/2]^{{2}m}.
\end{eqnarray*}

Second, if $ \frac1{64} a[ (1-a)/2]^{m-k }
<r_B\le \frac1{64}   a[ (1-a)/2]^{m-k-1}$ for some $1\le k\le m$,
 there are at most $2^{kn}$, up to a constant multiplier, many $B_{m,j} $ such that
 $ B\cap B_{m,j} \ne\emptyset$, and hence
\begin{eqnarray*}\Phi_\az(u_m,\,2B)
 &&\ls [(1-a)/2]^{(m-k)(2\az-n)} 2^{kn}  2^{-\ell(n+2-2\az)} [(1-a)/2]^{m(n+2-2\az)}.
\end{eqnarray*}
Since $2^{n}[ (1-a)/2]^{(n-2\az) }>1$ due to $ n-(n-2\az)\log[2/(1-a)]>0$,
we obtain
$$
\Phi_\az(u_m,\,2B)
 \ls 2^{mn}  2^{-\ell(n+2-2\az)} [(1-a)/2]^{m(n+2-2\az)}.
$$

Third, if $r_B> \frac1{64}   a[ (1-a)/2] $, there are at most $2^{mn}$, up to a constant multiplier, many $B_{m,j} $ such that
 $ B\cap B_{m,j} \ne\emptyset$, and hence
$$
\Phi_\az(u_m,\,2B)\ls 2^{mn}  2^{-\ell(n+2-2\az)} [(1-a)/2]^{m(n+2-2\az)}.
$$
To sum up, one obtains

$$\|u_m\|_{Q_\az(\rn)} \ls \max\{2^{-2\ell } [ (1-a)/2]^{2m}, 2^{mn} 2^{-\ell(n+2-2\az)} [(1-a)/2]^{m(n+2-2\az)}\}.$$
So \eqref{e5.1} will follow from this if one can show
$$
2^{-2\ell } [ (1-a)/2]^{2m}\le2^{mn} 2^{-\ell(n+2-2\az) } [ (1-a)/2]^{m(n+2-2\az)}.
$$
Obviously, this is equivalent to
$$
2^{\ell(n-2\az) }\le 2^{mn}  [ (1-a)/2]^{m(n -2\az)},
$$
and hence to
$$
\ell(n-2\az)\le mn+m(n-2\az)\log[(1-a)/2].
$$
But this last estimate follows from our choice of $\ell$, namely,
$$
\ell =\frac{mn}{n-2\az} \min\{\bz,\,\bz_0\}\le   \frac{mn}{n-2\az} \bz_0  = \frac{mn} {n-2\az} + m\log\lf(\frac{1-a}2\r).
$$
Thus \eqref{e5.1} holds.

{\bf Proof of \eqref{e5.2}.} Indeed, we have
\begin{eqnarray*}
\|u_m\circ f\|^2_{Q_\az (\rn)}&&\ge\Phi_\az (u_m\circ f,\,f^{-1}(B(0,\,2)))
 \\
&&\gs \sum_{j=1}^{2^{mn}} \int_{f^{-1}(B_{m,\,j})} \int_{f^{-1}(B_{m,\,j})} \frac {|u_{m,j}\circ f(x)-u_{m,j}\circ f(y)|^2} {|x-y|^{n+2\az}}\, dx\, dy\\
&&\gs \sum_{j=1}^{2^{mn}}   |f^{-1}(B_{m,j})|^{1-2\az/n}  \Phi_\az(u_{m,j}\circ f,\,f^{-1}(B_{m,j})).
\end{eqnarray*}

It suffices to estimate
$$
|f^{-1}(B_{m,j})|\ \ \&\ \ \Phi_\az(u_{m,j}\circ f,\,f^{-1}(B_{m,j}))
$$
from below. We first notice that if
$|x-z_{m,\,j}|<\frac12 a[(1-a)/2]^{m}
$,
then
 $$f^{-1}(x)=\lf(\frac12 a[(1-a)/2]^{m}\r)^{\bz/(\bz+1)}|x-z_{m,\,j}|^{-\bz/(\bz+1)}(x-z_{m,\,j})+z_{m,\,j}$$
and hence
 $$J_{f^{-1}}(x)\sim\lf(\frac12 a[(1-a)/2]^{m}\r)^{n\bz/(\bz+1)}|x-z_{m,\,j}|^{-n\bz/(\bz+1)}.$$
For every $y\in B_{m,j}$, we have that if  $|x-z_{m,\,j}|=2r_{m,j}
$,
then
$$  r_{m,\,j}\le |z_{m,\,j}-y_{m,j}|-|y-y_{m,j}|\le
|y-z_{m,\,j}|\le |z_{m,\,j}-y_{m,j}|+|y-y_{m,j}|\le 3 r_{m,\,j}
$$
and hence
$$
J_{f^{-1}}(y)\sim
\lf(\frac12 a[(1-a)/2]^{m}\r)^{ n\bz/(\bz+1)} r_{m,j}^{-n\bz/(\bz+1)}\sim
 2^{\ell n\bz/(\bz+1)}  .
$$
Therefore,
\begin{equation*}\label{e5.5}|f^{-1}(B_{m,\,j})|\sim   2^{-\ell n /(\bz+1)} [(1-a)/2]^{mn}
\end{equation*}
and
\begin{equation*}\label{e5.6}
\bint_{B_{m,j}}J_f(y)\,dy\sim  2^{\ell n\bz/(\bz+1)} \ls \einf _{y\in B_{m,j}} J_f(y).
\end{equation*}

Moreover, by $J_{f^{-1}}\in A_1(\rn)$ and similarly to \eqref{e4.2}, we have
\begin{eqnarray*}\Phi_\az(u_{m,j}\circ f,\,f^{-1}(B_{m,j}))&&\gs\Psi_{\az,\,2}(u_{m,j}\circ f,\,f^{-1}(2^{-4}B_{m,j})) \\
&&
\gs \Psi_{\az,\,2}(u_{m,j},\,2^{-4-N_2} B_{m,j})\\
&& \gs
\Phi_\az(u_{m,j},\,2^{-8-N_2}B_{m,j}).
\end{eqnarray*}
Notice that for all
$$
x \in  2^{-12-N_2}B_{m,j}\ \ \&\ \ y\in 2^{-8-N_2}B_{m,j}\setminus 2^{-9-N_2}B_{m,j},
$$
we have
$$
|x-y|\sim r_{m,j}\ \ \&\ \
|u_{m,j}(x)-u_{m,j}(y)|\ge  2^{-9-N_2}r_{m,j}-2^{-12-N_2}r_{m,j}\ge  2^{-10-N_2} r_{m,j}.
$$
Hence,
\begin{eqnarray*}
 \Phi_\az(u_{m,j},\,2^{-8}B_{m,j})
 && \gs r_{m,\,j}^{2\az-n} \bint_{2^{-12-N_2}B_{m,\,j}} \int_{2^{-8-N_2}B_{m,\,j}\setminus 2^{-9-N_2}B_{m,\,j}} \frac {|u_{m,j}(x)-u_{m,j}(y)|^2} {|x-y|^{n+2\az}} dx\, dy\\
 && \gs r_{m,\,j}^{2\az-n}  r_{m,j}^{2n} r_{m,j}^{-n-2\az+2}\\
 && \gs r_{m,\,j}^{2}.
\end{eqnarray*}
Therefore
\begin{equation}\label{e5.x1}\Phi_\az(u_{m,j}\circ f,\,f^{-1}(B_{m,j}))\gs r_{m,\,j}^{2}.
\end{equation}
This together with \eqref{e5.3} implies that
\begin{eqnarray*}
\|u_m\circ f\|^2_{Q_\az (\rn)}
&&\gs 2^{mn} 2^{  - \ell (n-2\az) /(\bz+1)}  [ (1-a)/2]^{m(n-2\az)} 2^{- 2\ell }[ (1-a)/2]^{2m}\\
&&\sim 2^{mn} 2^{  - \ell (n-2\az) /(\bz+1)} 2^{- 2\ell } [ (1-a)/2]^{m(n-2\az+2)}
\end{eqnarray*}
as desired.
 \end{proof}

 \begin{proof} [Proof of Theorem \ref{t5.1}]
 Fix $\az_0\in(0,\,1)$.
Let $\theta=(n-2\az_0)/n\in(0, 1)$ and  $\wz E_{\az_0}=(2^{\nn_\theta})^n$ be as  \eqref{e2.y1}  in Section 2.
By Lemma \ref{lemma 2.2}, $\overline\dim_{LG}(2^{\nn_\theta})^n=n-2\az_0$ but $ \dim_{L}(2^{\nn_\theta})^n= 0.$

Now we need to construct
a quasiconformal (Lipschitz) mapping $f:\rn\to\rn$ such that
 $J_f\in A_1(\rn,\,(2^{\nn_\theta})^n)$ but ${\bf C}_f$ is unbounded on $Q_{\az}(\rn)$ for each $\az \in(\az_0 ,1)$.
  The idea is similar to the construction of Theorem \ref{t1.4}.
 We divide the argument into two cases.

{\bf Case 1}:\ $\az_0=1$. 
Let $\bz>0$ and define
$$
\begin{cases}
f(x)=|x-\vec k|^\bz (x-\vec k)+\vec k\ \ \hbox{if}\ \ x\in B(\vec k,\,1)\ \ \hbox{with}\ \ \vec k\in (3\nn)^n;\\
f(x)=x\ \ \  \hbox{if}\ \ x\notin\cup_{\vec k\in\nn^n}B(\vec k,\,1).
\end{cases}
$$
Then $f$ is a quasiconformal mapping and
$$
\begin{cases}
J_f(x)\sim |x-\vec k|^{n\bz}\ \ \hbox{if}\ \ x\in B(\vec k,\,1)\ \ \hbox{for\ some}\ \ \vec k\in (3\nn)^n;\\
J_f(x)=1\ \ \hbox{otherwise}.
\end{cases}
$$

Now we show that ${\bf C}_f$ is unbounded on $Q_\az(\rn)$ for each $\az\in(0,\,1)$. Indeed, for each $\vec k\in (3\nn)^n$, we take a ball $B_{\vec k}$ such that
$$
|x_{B_{\vec k}}-\vec k|=2^{-m}\ \ \&\ \ r_{B_{\vec k}}=2^{-m-5}.
$$
Set
$$
u_{\vec k}(x)=\chi_{B_{\vec k}} d(x,\,\partial B_{\vec k}).
$$
For each $m$, set
$$
u_m=\sum_{|\vec k|\le 2^\ell}u_{\vec k}\ \ \hbox{with}\ \ \ell =m(n-2\az)/2\az.
$$
Observe that if $x \in B_{\vec k}$, then
  $$f^{-1}(x)= |x- {\vec k}|^{-\bz/(\bz+1)}(x-x_{\vec k})+x_{\vec k}
  $$
and hence
 $$J_{f^{-1}}(x)\sim |x- {\vec k}|^{-n\bz/(\bz+1)}\sim 2^{m\bz/(\bz+1)}.$$
 Thus, one gets
 $
 |f^{-1}(B_{\vec k})|\sim 2^{ -[1-\bz/(\bz+1) ]mn}.
 $

By an argument similar to \eqref{e5.x1} for $\Phi_\az(u_{m,j}\circ f,\,f^{-1}(B_{m,j}))$, we have
$$
\Phi_\az(u_{\vec k}\circ f,\,f^{-1}(B_{\vec k}))\gs 2^{-2m}.
$$
This leads to
\begin{eqnarray*}
\|u_m\circ f\|^2_{Q_{\az}(\rn)}
&&\ge\Phi_\az(u_m\circ f,\,f^{-1}(B(0,\,2^{\ell+1})))\\
&&\gs 2^{\ell(2\az-n)}\sum_{|\vec k|\le 2^\ell}
|f^{-1}(B_{\vec k})|^{1-2\az/n}\Phi_\az(u_{\vec k}\circ f,\,f^{-1}(B_{\vec k}))\\
&&\gs 2^{\ell(2\az-n)}\sum_{|\vec k|\le 2^\ell}
2^{-2m} 2^{-[1-\bz/(\bz+1)] m(n-2\az) } \\
&&\gs  2^{2\az\ell}2^{-2m} 2^{-[1-\bz/(\bz+1)] m(n-2\az) }\\
&&\gs 2^{-2m}2^{ m(n-2\az)\bz/(\bz+1)  },
\end{eqnarray*}
where
$
\ell =m(n-2\az)/2\az.
$

On the other hand, we claim that
$
\|u_m\|^2_{Q_{\az}(\rn)} \ls 2^{-2m}.
$
The proof of this estimate is similar to that of  \eqref{e5.1} . Five situations are required to handle.

If $r_B\le 2^{-m-5}$, by an argument similar to \eqref{e5.3}, we have
 $\Phi_{ \az}(u_m,\, 2B ) \ls 2^{-2 m}.
 $

If $r_B> 2^{-m-5}$, similarly to \eqref{e5.4}, we also have
 $$
 \Phi_{ \az}(u_m,\, 2B ) \ls |B|^{2\az/n-1}\sum_{B_{\vec k}\cap 2B\ne\emptyset} 2^{- m(2-2\az+n)}.
 $$

If $2^{-m-5}<r_B\le1$, there is at most one $B_{\vec k} $ such that
 $ B\cap B_{\vec k} \ne\emptyset$
 and hence
 $\Phi_{ \az}(u_m,\, 2B ) \ls 2^{-2 m }.$

If $1\le r_B\le 2^\ell$, then
 there are at most $2^{n+2}r_B^n$ many $B_{\vec k} $ such that
 $ B\cap B_{\vec k} \ne\emptyset$, and hence
 $$\Phi_{ \az}(u_m,\, 2B ) \ls r_B^n  r_B^{2\az-n} 2^{- m(2-2\az+n)}\ls 2^{2\az \ell}2^{- m(2-2\az+n)}\ls 2^{-2m},$$
where $\ell =m(n-2\az)/2\az$.

If $ r_B> 2^\ell$, then
 there are at most $2^{n+2}2^{\ell n }$ many $B_{\vec k} $ such that
 $ B\cap B_{\vec k} \ne\emptyset$, and hence
 $$\Phi_{ \az}(u_m,\, 2B ) \ls    r_B^{2\az-n} 2^{\ell n}2^{- m(2-2\az+n)}\ls 2^{2\az \ell}2^{- m(2-2\az+n)}\ls 2^{-2m},$$
where $\ell =m(n-2\az)/2\az$.

Finally, we have
 $$\frac{\|u_m\circ f\|_{Q_{\az}(\rn)}}{\|u_m \|_{Q_{\az}(\rn)}}
\gs 2^{  m( n-2\az)\bz/(\bz+1)}
\to\fz$$ as $m\to\fz$  since $\bz>0$.

{\bf Case 2}:\ $\az_0\in(0,\,1)$.
Similarly to {\bf Case 1}: $\az_0=1$, we   can first construct
 quasiconformal mappings $f:\rn\to\rn$  with $J_f\in A_1(\rn,\,   (2^{\nn_\theta})^n)$,
and then construct the critical function $u_m$
similarly to {\bf Case 1}: $\az_0=1$, but the key parameter $\ell$ over there is now taken as
$
m(n-2\az)/(2\az-n+\theta n)
$
{where}
$$
2\az-n+\theta n>0\Leftrightarrow 2\az>n-\theta n =\az_0.
$$
Such a ${\bf C}_f$ is not bounded  on $Q_{\az}(\rn)$ for all $\az\in(\az_0,\, 1 )$ and hence satisfies our requirement; we omit the details.
\end{proof}

\medskip

\noindent {\bf Acknowledgements.} The third and forth authors would like to thank Professor Jacques Peyriere for kind discussions on the Minkowski type dimension.
Part of this research was done during the first and forth authors' visit  at IPAM, UCLA; both authors would like to thank for the support.


\end{document}